\documentclass[11pt]{amsart} 
\usepackage{hyperref}
  \usepackage{amsmath}
    \usepackage{amsthm}
    \usepackage{amssymb}
   \usepackage[dvipsnames]{xcolor}  
   \usepackage{graphicx}
    \usepackage{pdfsync}
    \usepackage{cleveref} 
\def\corON{\textcolor{purple}}
\def\corON{}
\usepackage{tikz}
    \usetikzlibrary{decorations.pathreplacing}
    \usetikzlibrary{decorations.pathmorphing}
\usetikzlibrary{snakes}
\usetikzlibrary{arrows.meta}

\def\corDB{}
\def\corOF{}

\newtheorem{theorem}{Theorem}[section]
\newtheorem{lemma}[theorem]{Lemma}
\newtheorem{corollary}[theorem]{Corollary}
\newtheorem{proposition}[theorem]{Proposition}
 
\newcommand{\eqnsection}{
   \renewcommand{\theequation}{\thesection.\arabic{equation}}
   \makeatletter
   \csname @addtoreset\endcsname{equation}{section} 
   \makeatother}

\def \be{\begin{equation}}
\def \ee{\end{equation}}
\def \bt{\begin{theorem}} 
\def \et{\end{theorem}}
\def \bl{\begin{lemma}} 
\def \el{\end{lemma}}
\def \bea{\begin{eqnarray}}
\def \eea{\end{eqnarray}}
\def \bas{\begin{eqnarray*}}
\def \eas{\end{eqnarray*}}

\renewcommand{\eqref}[1]{(\ref{#1})} 

\newcommand{\al}{\alpha} 
\newcommand{\bb}{\beta} 

\newcommand{\ga}{\gamma} 
\newcommand{\ep}{\epsilon} 
\newcommand{\vep}{\varepsilon}

\newcommand{\la}{\lambda}

\newcommand{\bE}{\mathbb{E}}

\newcommand{\bN}{\mathbb{N}}

\newcommand{\bP}{\mathbb{P}}
\newcommand{\bQ}{\mathbb{Q}}
\newcommand{\bR}{\mathbb{R}}

\newcommand{\bZ}{\mathbb{Z}}

\newcommand{\CC}{\mathcal{C}}

\newcommand{\FF}{\mathcal{F}}

\newcommand{\LL}{\mathcal{L}}

\newcommand{\TT}{\mathcal{T}}

\newcommand{\bs}{\mathbf{s}}

\newcommand{\wt}{\widetilde}
\newcommand{\wh}{\widehat}
\def \ff{\infty}

\DeclareMathOperator{\Var}{Var}

\DeclareMathOperator{\Cov}{Cov}

\def\b1{\mathbf 1}

\def \({\left(}
\def \){\right)}

\def \lc{\left\{}
\def \rc{\right\}}

\def \nn{\nonumber}

\def \bc{\begin{center} }
\def \ec{\end{center} }
\def \bs{\begin{slide} }
\def \es{\end{slide} }

\def\square{{\vcenter{\vbox{\hrule height.3pt
        \hbox{\vrule width.3pt height5pt \kern5pt
           \vrule width.3pt}
        \hrule height.3pt}}}}
\def\qed{{\hfill $\square$ \bigskip}}

\newcommand{\abbr}[1]{{\sc\lowercase{#1}}}

\eqnsection

\begin{document}
	
\title[Cover time of binary tree]{Limit law for the cover time\\
of a random walk on a binary tree}
\author[Amir Dembo\;\; Jay Rosen\;\; Ofer Zeitouni]
{Amir Dembo\;\;  Jay Rosen\;\; 
Ofer Zeitouni}

\date{June 16, 2019}

\thanks{Amir Dembo was partially supported by NSF grant DMS-1613091.}
\thanks{Jay Rosen was partially supported by the Simons Foundation.}
\thanks{Ofer Zeitouni was supported by the ERC advanced grant LogCorFields.}
\subjclass[2010]{60J80; 60J85; 60G50}
 \keywords{Cover time. Binary tree. Barrier estimates.}
\maketitle

\begin{abstract}
Let $\TT_n$ denote the binary tree of depth $n$ augmented by an extra edge connected to its root. Let  $\CC_n$ denote the cover time of $\TT_n$ by simple random walk. 
We prove that $\sqrt{ \mathcal{C}_{n}  2^{-(n+1) } } -   m_n$ converges in distribution as $n\to \infty$, where $m_n$ is an explicit constant, and identify the limit.
\end{abstract}

\section{Introduction}
We introduce in this section notation and our main results, provide background, and give a road map for the rest of the paper.

\subsection{Notation and main result}
Let $\TT_n$ denote the binary tree of depth $n$, whose only vertex of degree $2$  is  attached to an extra vertex 
$\rho$, called the  \abbr{root}.
The \abbr{cover time} $\CC_n$ of $\TT_n$ is
the number of steps of a (discrete time)  simple random walk started at $\rho$, till visiting 
all vertices of $\TT_n$. We write  \abbr{srw} for such a random walk. 
The main result of this paper is the following theorem, which gives convergence in law of a (normalized) version of the cover time $\CC_n$.
\begin{theorem}
\label{theo-covertime}
Let  $\CC'_n := 2^{-(n+1)} \CC_n$, and set
\begin{equation}
m_{n} : = \rho_{n} n\,, \quad
\rho_{n}: =  c_\ast-\frac{\log n}{c_\ast n}\,, \quad
 \quad c_\ast=\sqrt{2\log 2}\,.
\label{28.00}
\end{equation}
There exist a random variable $X'_\infty>0$ 
and $\alpha_{\ast}>0$ finite, so that, for any fixed $y\in \bR$,
\begin{align}
\lim_{n\to\ff} \bP \big(\sqrt{\CC'_n} -   m_n  \le y  \big)  =  \bE \(\exp\{-\alpha_{\ast} X'_\infty \mathrm{e}^{-c_{\ast}y}\} \) :=  \bP(Y'_\infty\leq y) \,.
\label{ct2.25b}
\end{align}
\end{theorem}
\noindent 
That is, the normalized cover time $\sqrt{\CC'_n} -   m_n$ converges in distribution to 
$Y'_\infty$, a standard Gumbel random variable
shifted by $\log (\alpha_{\ast} X'_\infty)$ then scaled by $1/c_\ast$.
See  \Cref{lem-dermartconv} for a description of the random variable
$X'_\infty$ in terms of  the limit of a derivative martingale.

As is often the case, most of the work in proving a statement such as Theorem \ref{theo-covertime} 
involves the control of certain excursion counts. We now introduce notation in order to describe these.  
Let  $V_j$ denote the set of vertices of $\TT_n$ at level $j$, with $V_{-1}=\{\rho\}$. 
For each $v\in V_{n}$ let $v(j)$ denote the ancestor of $v$ at level $j \le n$, 
i.e. the unique vertex in $V_j$ on the geodesic connecting 
$v$ and \corDB{$\rho$}. In particular, $v(-1)$ is  the root \corDB{$\rho$}.

Next, for each $v\in V_{n}$ let 
\begin{align}
T_{v,j}^{s} = \#\{ &\mbox{\rm excursions from  $v(j-1)$ to $v(j)$ made by the  \abbr{srw} on $\TT_n$}\nonumber \\
 &\; \mbox{\rm prior to completing its first $s$ excursions from the root $\rho$}\} .
 \label{Tjsdef}
\end{align}
Setting
\begin{equation}
t^{\ast}_{v,n}=\inf  \{s \in \bZ_+\,:\, T_{v,n}^{s} \neq 0\}\,, 
\label{deftast}
\end{equation}
for the number of excursions to the root by the \abbr{srw} 
till reaching $v$, we consider the corresponding excursion 
cover time of $V_{n}$,
\begin{equation}\label{eq-tsar}
t^{\ast}_{n}:=\sup_{v\in V_{n}} \{ t^{\ast}_{v,n} \} \,.
\end{equation}



Our main tool in proving \Cref{theo-covertime} is a generalization (see \Cref{prop-eta-tau} below),
of the following theorem concerning $t_n^*$.
\begin{theorem}
\label{thm1.2}
With notation as in Theorem \ref{theo-covertime}, 
\begin{equation}
\sqrt{2 t^{\ast}_n} - m_n \overset{dist}{\Longrightarrow}  Y_\infty  
\qquad \mbox{ as } \quad n\to\ff \,,\label{new-CT.0}
\end{equation}
where $Y_\ff := Y'_\ff -  \bar g_\infty$ for a standard Gaussian random variable $\bar g_\infty$,
independent of $Y'_\infty$. Alternatively, for some random variable  $X_\infty>0$,
\begin{equation} 
\bP(Y_\infty\leq y) :=  \bE \(\exp\{-\alpha_{\ast} X_\infty \mathrm{e}^{-c_{\ast}y}\} \) \,. \label{ct2.25}
\end{equation}
\end{theorem}
\noindent
As we will see, most of the technical work in the proof of \Cref{thm1.2} (or
 \Cref{prop-eta-tau}), is in obtaining the sharp tail estimates of \Cref{prop-limiting-tail-gff} below.

To describe  $X_\infty$ and $X'_\infty$,
let $\{g_{u},  u \in \TT_\infty \}$ be the standard Gaussian branching 
random walk (\abbr{brw}), on the infinite binary tree $\TT_\infty$. That is, 
placing i.i.d. standard normal weights on the edges of $\TT_\infty$,
we write  $g_{u}$ for the sum of the weights along the geodesic connecting  $0$ to $u$.
We further consider the empirically centered  $g_u':=g_u - \bar g_{|u|}$, 
where 
\begin{equation}\label{barg-k}
\bar g_k = 2^{-k} \sum_{u'  \in V_k} g_{u'} \,, \quad k \in \bZ_+ \,,
\end{equation}
 denotes the average of the 
\abbr{brw} at level $k$, and set
\begin{equation}
X_k = \sum_{u \in V_k}\(   c_\ast k+ g_u \) \,\,\mathrm{e}^{-c_\ast\(   c_\ast k+ g_u \)}
\,, \;
X'_k = \sum_{u \in V_k} \( c_\ast k + g'_u \) \,\, \mathrm{e}^{-c_\ast\( c_\ast k + g'_u \)}\,.
 \label{dmart.1}
\end{equation} 
It is not hard to verify that $\{X_k\}$ is a martingale, referred to as the \abbr{derivative martingale}.
We then have that
\begin{lemma} 
\label{lem-dermartconv} 
$X_k$, $X'_k$ and $\bar g_k$ converge a.s. to positive, finite limits $X_\infty$, 
$X'_\infty$ and a standard Gaussian variable $\bar g_\infty$, independent of $X'_\infty$, such that 
\begin{equation}\label{tilted-dermg}
X_\infty = X'_\infty e^{-c_\ast \bar g_\infty} \,.
\end{equation}
\end{lemma}
The convergence of $X_k$ to $X_\infty$ is well known, see e.g. \cite{Aidekon} (and, for 
its first occurrence in terms of limits in branching processes, \cite{LalleySelke}), and building 
on it, we  easily deduce the corresponding convergence for $X'_k$.

\subsection{Background and related results}
Theorem \ref{theo-covertime} is closely related to the recent paper \cite{CLS}, which deals with continuous time \abbr{SRW}, and we wish to acknowledge priority
to their work. The proofs however are different - while \cite{CLS} builds heavily on the isomorphism theorem of 
\cite{EKMRS} to relate directly the occupation time 
on the tree to the Gaussian free field on the tree which is nothing but the \abbr{BRW}  described above
Lemma \ref{lem-dermartconv}, our proof  is a refinement of \cite[Theorem 1.3]{BRZ}, where 
the tightness of the \abbr{lhs} of \eqref{new-CT.0} is proved. Our proof, which is based
 on the strategy for proving convergence in law of the maximum of branching random walk described in \cite{BDZconvlaw},
  was obtained independently of \cite{CLS}, except that in proving that Theorem \ref{thm1.2} implies Theorem \ref{theo-covertime}, we do borrow  some ideas from \cite{CLS}.
  As motivation to our work, we note that estimates from \cite{BRZ} were instrumental in obtaining 
 the tightness of  the (centered) cover time of the two dimensional sphere by $\epsilon$-blowup of Brownian motion, see \cite{BRZ17}. We expect that
 the ideas in the current work will play an important role in improving the tightness result of 
 \cite{BRZ17} to convergence in law. We defer this to forthcoming work.
 
 We next put our work in context. The study of the cover time of graphs by \abbr{SRW} has a long history.
Early bounds appear in \cite{Matthews}, and a general result showing that the cover time is concentrated as soon as it
is much longer than the maximal hitting time appears in \cite{Aldousgeneral}.
A modern general perspective linking the cover time of graphs to Gaussian processes appears in \cite{DLP},
and was refined to sharp concentration  in \cite{ding} (for many graphs including trees) 
and \cite{zhai} (for general graphs). See also \cite{lehec} for a different perspective on \cite{DLP}. For the cover time of trees, 
an exact first order asymptotic appears in \cite{Aldous}. The tightness of $\sqrt{ \CC'_{n} }$ around an implicit constant was derived by analytic methods in 
\cite{BZ}, and, following  the identification of the logarithmic correction in $m_n$  \cite{DingZeitouni-ASharpEstimateForCoverTimesOnBinaryTrees}, its $O(1)$ identification
appears in \cite{BRZ}.  

We note that the evaluation of the cover time is but one of many natural questions concerning the 
process of points with a-typical (local) occupation time, and quite a bit of work has been devoted to this topic. 
We do not elaborate here and refer the reader to \cite{DPRZ1,Sznitman,AbeBiskup}. 
Particularly relevant to this paper is the recent 
\cite{Abe}.

It has been recognized for quite some time that the study of the cover time of two dimensional manifolds by Brownian motion (and of the
cover time of two dimensional lattices by \abbr{SRW}) 
 is related to a hierarchical structure similar to that appearing in the study of the cover time for trees,
 see e.g. \cite{DemboPeresEtAl-CoverTimesforBMandRWin2D} and, for a recent perspective, \cite{Schmidt}.  A similar 
hierarchical structure also appears in the study of extremes of the critical  Gaussian free field, and in other logarithmically correlated fields appearing e.g. in the study of random matrices. 
We do not discuss that literature and refer instead to  recent surveys offering different perspectives
\cite{Arguin,Biskup,Bovier,Kistler,Zeitouni}.

\subsection{Structure of the paper}
In contrast with \cite{CLS} the key to our proof of \Cref{thm1.2} is the following 
sharp right tail for the excursion cover times.
\begin{theorem}\label{prop-limiting-tail-gff}
There exists a finite $\alpha_\ast>0$ such that  
\begin{equation}
\label{righttail.1}
\lim_{z\to \infty}\limsup_{n\to \infty}|z^{-1} \mathrm{e}^{c_\ast z} 
\bP (\sqrt{2t^{\ast}_n} -m_n >z) - \alpha_\ast| =0\,.
\end{equation}
\end{theorem}
After quickly dispensing of \Cref{lem-dermartconv}, in Section \ref{sec-cl} 
we obtain \Cref{thm1.2} out of  \Cref{prop-limiting-tail-gff}, by adapting the 
approach of \cite{BDZconvlaw} for the convergence in law 
of the maximum of \abbr{brw}. The main difference is that here we have 
a more general Markov chain (and not merely a sum of i.i.d.-s). 

In the short \Cref{sec-ctime}, which is the only part of this work that parallels
the derivation of \cite{CLS}, we deduce  \Cref{theo-covertime} out of \Cref{thm1.2} 

The bulk of this paper is devoted to the proof of \Cref{prop-limiting-tail-gff}, which 
we establish in \Cref{sec-rt} by a refinement of the approach used in 
deriving \cite[Theorem 1.3]{BRZ}. In doing so, we defer the a-priori 
bounds we need on certain barrier events,
which might be of some independent interest, to \Cref{sec-re},
where we derive these bounds by refining estimates from \cite{BRZ}.
The proof of the main contribution to the tail estimate of 
\Cref{prop-limiting-tail-gff}, as stated in \Cref{prop-asymptotic-first-moment}, is further 
deferred to   \Cref{sec-sharpbar}. There, utilizing 
the close relation between our Markov 
chain and the $0$-dimensional Bessel process, we get 
 sharper barrier estimates, now up to $(1+o(1))$ factor
of the relevant probabilities.

\section{From tail to limit:  \Cref{lem-dermartconv} and  \Cref{thm1.2}}\label{sec-cl}
\label{sec-maintheorem}

We start by proving the elementary \Cref{lem-dermartconv}, denoting throughout 
the last common ancestor of $u,u' \in \TT_\infty$ by $w = u \wedge u'$.
Namely, $w=u(|w|)$ for $|w|=\max \{ j \ge 0,u(j)=u'(j)\}$. 
\begin{proof}[Proof of \Cref{lem-dermartconv}]
The \abbr{brw} $\{g_{u};\,\, u \in \TT_k  \}$ of
\Cref{lem-dermartconv} is the centered Gaussian random vector having 
\begin{equation}\label{cov-brw}
\Cov(g_u,g_{u'}) = |u \wedge u'| \,.
\end{equation}
Further, the average of the \abbr{brw} weights on the edges of $\TT_\infty$ 
between levels $(k-1)$ and $k$, is precisely $\Delta g_k := \bar g_k - \bar g_{k-1}$ for $\bar g_k$ of \eqref{barg-k}.
With $(\Delta g_k, k \ge 1)$ independent centered Gaussian random 
variables with $\Var(\Delta g_k) = 2^{-k}$, we have that $\bar g_k$ 
converges a.s. to the standard Gaussian $\bar g_\infty := \sum_k \Delta g_k$.  
Next, recall the existence of  $w_k \to \infty$ such that, a.s.,
\begin{equation}\label{Ak-def}
A_k := \{ c_\ast k + g_u  \in w_k + (0,2 c_\ast k)  \,, \;\; \forall u \in V_k \}\;  \mbox{\rm occurs for all $k$ large,}
\end{equation}
see e.g. \cite[(1.8)]{AS09}. For $X_k$ of \eqref{dmart.1}, it follows from \cite{Aidekon} that
$X_k \stackrel{a.s.}{\to} X_\infty \in (0,\infty)$,
while 
\begin{equation}\label{tildeXk}
X_k > w_k \widetilde{X}_k   \quad \mbox{on}  \;\; A_k, \qquad \mbox{for} \qquad 
\widetilde X_k := \sum_{u \in V_k} e^{-c_\ast (c_\ast k + g_u)} \,.
\end{equation}
Thus,  $\widetilde X_k \stackrel{a.s.}{\to} 0$ as $k \to \infty$. From the two 
expressions in \eqref{dmart.1} we have  that 
\[
X'_k = (X_k - \bar g_k \widetilde X_k) e^{c_\ast \bar g_k} 
\]
which thereby converges a.s. to $X'_\infty = X_\infty e^{c_\ast \bar g_\infty}$ as claimed 
in \eqref{tilted-dermg}. Finally, from \eqref{cov-brw} we deduce that for any $u \in V_k$, $k \ge 0$,
 \begin{align}
\Cov(g_u, \bar g_k)&
 =2^{-k}  \sum_{u' \in V_k} |u \wedge u'| 
 = \sum_{j=1}^{k}  (j-1) 2^{-j} + k 2^{-k} = 1-2^{-k}.
 \label{ct10j}
 \end{align}
This covariance is constant over $u \in V_k$, hence
$\Cov(g'_u,\bar g_{|u|}) = 0$ for $g'_u := g_u - \bar g_{|u|}$, implying 
the independence of $\bar g_k$ and $\{ g'_u, u \in V_k \}$. The latter 
variables are further independent of the  \abbr{brw} edge weights outside $\TT_k$, 
hence of $\bar g_\infty$. Thus, the random variable $X'_\infty$, which is measurable on
$\sigma(g'_u, u \in \TT_\infty)$, must also be independent
of $\bar g_\infty$. 
\end{proof}
\noindent
We next normalize the counts $T^s_{u,j}$ of \eqref{Tjsdef} and define
\begin{equation}\label{std-occ}
\wh T_{u}(s) := \frac{T_{u,|u|}^{s} - s}{\sqrt{2s}} \,, \qquad \wh{S}_k(s) := 2^{-k} \sum_{u \in V_k} \wh{T}_u(s)\,,
\end{equation}
and get from the \abbr{clt}  for sums of i.i.d.
the following relation with the \abbr{brw}.
\bl \label{lem-wish} 
For fixed $k$ and the \abbr{brw} $\{g_{u};\,\, u \in \TT_k  \}$ of \Cref{lem-dermartconv}, we have    
\begin{align}\label{clt-brw}
\lc \wh{T}_{u}(s) ,\, u\in \TT_{k}\rc & \overset{dist}{\underset{s\to \ff}{\Longrightarrow}}
\lc g_u \,,\, u \in \TT_k \rc, \\
\lc \wh{T}_{u}(s) - \wh{S}_k (s) ,\, u\in V_{k}\rc & \overset{dist}{\underset{s\to \ff}{\Longrightarrow}}
\lc g'_u \,,\, u \in V_k \rc.
\label{clt-mbrw}
\end{align}
\el 
\begin{proof}  The consecutive excursions to $\rho$ by the \abbr{srw} on $\TT_n$ are i.i.d. 
Hence, $s \mapsto \{T^s_{u,|u|},\; u \in \TT_k\}$ is an $\bR^d$-valued random walk (with $d
$ 
the finite size of $\TT_k$). Further, projecting the \abbr{srw} on $\TT_n$ to the geodesic from $u$  
to $\rho$, yields a symmetric \abbr{srw} on $\{-1,0,\ldots,|u|\}$. Thus, denoting by 
$T_j$ the number of excursions from $u(j-1)$ to $u \in V_j$ during a single excursion 
to $\rho$, 
we have that  
$\bP(T_j \ge 1) = {\sf p}_j := 1/(j+1)$ (for reaching $u$ before returning to $\rho$), 
and $T_j$ conditional on $T_j \ge 1$, follows a geometric law 
of success probability $\bP(T_j=1 |T_j \ge 1)={\sf p}_j$.
Consequently, for any $j \in [0,k]$,
 \begin{equation}
 \begin{aligned}
 \bE (T_j)  = 1 \,, \qquad
  \label{ct2.6}
 \Var(T_{j}) &= \bE [ T_j (T_j-1) ] = \frac{2 (1-{\sf p}_j)}{{\sf p}_j} = 2 j \,.
\end{aligned} 
 \end{equation}
Note that $T^1_{u,|u|}$ and $T^1_{u',|u'|}$ are independent, conditionally on $T^1_{w,|w|}$, 
for $w=u \wedge u'$, each having the conditional mean $T^1_{w,|w|}$. We thus see that
for any $u,u' \in \TT_k$, in view of \eqref{ct2.6},  
 \begin{equation}
\Cov(T^1_{u,|u|},T^1_{u',|u'|}) = \Var(T_{|u \wedge u'|}) =2 |u\wedge u'|  \,. \label{ct2.8}
 \end{equation}  
Comparing with \eqref{cov-brw}, the i.i.d. increments of our $\bR^d$-valued 
random walk have the mean vector ${\bf 1}$ and covariance matrix which is twice
that of the \abbr{brw}, with  \eqref{clt-brw} and \eqref{clt-mbrw} as immediate consequences of the  multivariate \abbr{clt}.
\end{proof}

\noindent
Using throughout the notation 
\begin{equation}
\label{eq:defs}
s_{n,y} := (m_n+y)^2/2
\end{equation}  for $m_n$ of \eqref{28.00}, we have 
that  
\[
\{ \sqrt{2 t^\ast_n} - m_n \le y\} = \{ t^\ast_n \le s_{n,y} \}.
\]
In view of \Cref{lem-wish}, we thus see that  \Cref{thm1.2} is an immediate 
consequence (for non-random $\tau_k(s)=s$), of  \eqref{ct2.17} in  the following lemma. (The additional statement 
employing  \eqref{ct2.12} is utilized in the proof of \Cref{theo-covertime}.)
\begin{proposition}\label{prop-eta-tau}
Let $\FF_k$ denote the $\sigma$-algebra of the $\TT_k$-projection of the 
\abbr{srw} on $\TT_n$.  If  $\FF_k$-measurable $\{ \tau_k(s), s \ge 0\}$ are such that  
\begin{equation}
\wh{\tau}_k(s) := \Big(\frac{\tau_k(s) -s}{\sqrt{2s}}\Big) 
\overset{p}{\underset{s\to \ff}{\longrightarrow}} 0,
 \label{ct2.12b}
 \end{equation}
then for any fixed $y \in \bR$,
\begin{equation}
\label{ct2.17}
\lim_{k\rightarrow\infty} \limsup_{n \to \infty}
|\bP( t^{\ast}_{ n}   \le  \tau_{k}(s_{n,y}) ) -  \bP(Y_\infty \le y)| = 0 \,.
\end{equation}
Further, replacing \eqref{ct2.12b} by
 \begin{equation}
\wh{\tau}_k(s) + \wh{S}_k(s) 
\overset{p}{\underset{s\to \ff}{\longrightarrow}} 0 \,,
\label{ct2.12}
\end{equation}
leads to \eqref{ct2.17} holding with $Y'_\infty$ of \eqref{ct2.25b} instead of $Y_\infty$.
\end{proposition}
\begin{proof} For a possibly random, $\FF_k$-measurable $\tau$,  we set 
\begin{equation}\label{def:T-hat}
\wh{T}_u(\tau;s) := \frac{T_{u,|u|}^{\tau} - s}{\sqrt{2s}} \,,  
\end{equation}
in analogy to $\wh{T}_u(s;s)=\wh{T}(s)$ of  \eqref{std-occ}. In case
 $\frac{\tau_k(s)}{s} \stackrel{p}{\to} 1$ as $s \to \ff$, we get from Donsker's invariance principle that 
\begin{equation}\label{inv-brw}
\wh{T}_u(\tau_k(s);s) - \wh{T}_u(s;s)  - \wh{\tau}_k(s) 
\overset{p}{\underset{s\to \ff}{\longrightarrow}} 0\,,  \qquad \forall u \in \TT_k \,.
\end{equation}
Further,  \corOF{
\begin{equation}\label{eq:f_s}
f_s(x) := \sqrt{ 2(s+\sqrt{2s} x)} - \sqrt{2s} \overset{}{\underset{s\to \ff}{\longrightarrow}} x \,,
\end{equation}}
uniformly over bounded $x$. Hence, setting 
\begin{equation}\label{def:T-tilde}
\wt{T}_u(\tau;s) := \sqrt{2 T_{u,|u|}^{\tau}} - \sqrt{2s} = f_s( \wh T_u (\tau;s) ) \,,
\end{equation}
upon combining  \Cref{lem-wish} and \eqref{inv-brw}, we deduce from  \eqref{ct2.12b} that 
\begin{equation}\label{cnv-sqrt-tau-brw}
\Big\{ \wt{T}_u(\tau_k(s);s)  ,\, u\in V_k \Big\} 
\overset{dist}{\underset{s\to \ff}{\Longrightarrow}}
\lc g_u \,,\, u \in V_k \rc,
\end{equation}
whereas under \eqref{ct2.12} we merely replace $g_u$ by $g'_u$ on the \abbr{rhs}. Proceeding 
under the assumption \eqref{ct2.12b}, 
fix $y \in \bR$ and an integer $k \ge 1$, setting  
\begin{equation}\label{eq:zv-def}
z_u^{(n)} := \sqrt{2 T_{u,k}^{\tau_k(s_{n,y})}} - m_{n-k} \,, \qquad z_u^{(\infty)} := c_\ast k + g_u + y,\,  
\quad \forall u \in V_k \,, 
\end{equation}
with  $c_\ast$ as in \eqref{28.00}. 
For fixed $y$ and $k$, we have, using \eqref{eq:defs}, that
\[
\sqrt{2 s_{n,y}} - m_{n-k} - (c_\ast k + y) =
m_n - m_{n-k} -  c_\ast k = \frac{1}{c_\ast} \log (1-\frac{k}{n})_{ \stackrel{\longrightarrow}{\small{n\to\infty}}} 0\,. 
\]
Hence from \eqref{cnv-sqrt-tau-brw} at $s=s_{n,y}$ it follows that 
\begin{equation}\label{zvn-conv}
\{z_u^{(n)}, \; u \in V_k \}  \overset{dist}{\underset{n \to \ff}{\Longrightarrow}} \{z_u^{(\infty)}, \; u \in V_k \} .
\end{equation}
In particular, for $X_k$ of \eqref{dmart.1} and $\widetilde{X}_k$ of \eqref{tildeXk}
 we have that  
\begin{equation}\label{Xkn-conv}
X_k^{(n,y)} := \sum_{u \in V_k} z_u^{(n)} e^{-c_\ast z_u^{(n)}}  \overset{dist}{\underset{n \to \ff}{\Longrightarrow}}
( X_k + y \widetilde{X}_k ) e^{-c_\ast y} \,. 
\end{equation}
For any fixed $y \in \bR$,  we have by \eqref{Ak-def} and  \eqref{zvn-conv} that 
\begin{equation}\label{Ank-def}
\lim_{k \to \infty} \varliminf_{n \to \ff} \bP(A^{(n)}_k) = 1 \,, \;\; 
A^{(n)}_k = \{ z_u^{(n)}  \in  w_k + y + (0,2 c_\ast k) \,, \; \forall u \in V_k \} \,.
\end{equation}
Recalling that $\widetilde{X}_k \stackrel{a.s.} \to 0$ (see the line following \eqref{tildeXk}), and
the definition of $Y_\infty$ from \eqref{ct2.25}, we have 
in view of \eqref{Xkn-conv} that 
for any $\alpha_k \to \alpha_\ast$
\begin{align}\label{Yinf-conv}
\bP(Y_\infty\leq y) &= \lim_{k \to \ff}
\bE \Big[ {\bf 1}_{A_k}
\exp\{-\alpha_k (X_k + y \widetilde{X}_k) \mathrm{e}^{-c_{\ast}y}\}  \Big] \nn \\
& = \lim_{k \to \ff} 
{\mathop{\underline{\overline{\lim }}}_{n\to\ff}}
\bE \Big[  {\bf 1}_{A^{(n)}_k}  \exp\{-\alpha_k X^{(n,y)}_k \}  \Big]  \nn \\
& = \lim_{k \to \ff} 
{\mathop{\underline{\overline{\lim }}}_{n\to\ff}}
\bE \Big[  {\bf 1}_{A^{(n)}_k} \prod_{u \in V_k} \big(1-\alpha_k z_u^{(n)} e^{-c_\ast z_u^{(n)}} \big)
 \Big] \,,
\end{align}
where
${\mathop{\underline{\overline{\lim }}}_{n\to \ff}}f(n)$ 
stands for bounds given by both $\limsup_{n\rightarrow\infty} f(n)$
and $\liminf_{n\rightarrow\infty} f(n)$, and in the last equality of \eqref{Yinf-conv} we 
relied on having 
\[
\delta_k:= \sup_{n}  {\bf 1}_{A_k^{(n)}} \, \sup_{u \in V_k} \{ \alpha_k z_u^{(n)} e^{-c_\ast z_u^{(n)}} \}_{ \stackrel{\longrightarrow}{\small{k\to\infty}}} 0 \,,
\] 
as well as $e^{-a}\geq 1-a \geq e^{-a(1+\delta)}$ for $a\in [0,\delta \wedge 1/2]$.
For $u \in V_k$ let $V^u_n = \{ v \in V_n : v(k) = u  \}$ denote the leaves of the binary sub-tree of  
$\TT_n$ of depth $n-k$,
emanating from $u$, with $u(k-1)$ acting as its (extra) root. 
The event $\{t_n^\ast \le \tau\}$ of the \abbr{srw} reaching all of $V_n$ within its first $\tau$ excursions 
to $\rho$ is the intersection over $u \in V_k$ of the events of reaching all of $V^u_n$ 
within the first $T^{\tau}_{u,k}$ excursions of the \abbr{srw} between $u(k-1)$ and $u$. 
By the Markov property, for $\FF_k$-measurable $\tau$, conditionally on
$\FF_k$ 
the latter events are mutually independent, of conditional probabilities
$\bar \ga_{n-k}(T_{u,k}^{\tau})$ for $u \in V_k$ and
 $\bar \ga_n(s):=\bP(t^\ast_n \le s)$. 
Consequently, for $\tau=\tau_k(s_{n,y})$ we get that  
\begin{align}
\bP(t^{\ast}_n \le \tau |\,\FF_k) 
&= \prod_{u \in V_k}  \bar \ga_{n-k} (T_{u,k}^{\tau}) 
= \prod_{u \in V_k} \(1-\ga_{n-k}(z^{(n)}_u) \) \,, \quad 
\label{dmart.10}
\end{align}
for $\ga_n(z) := \bP   (\sqrt{2t^{\ast}_{n}} -m_n>z)$ and $z^{(n)}_u$ of \eqref{eq:zv-def}. 
\Cref{prop-limiting-tail-gff} and the monotonicity of 
$z \mapsto \ga_n(z)$
yield that for some $n_k<\infty$ and $\alpha^{(\pm)}_k \to \alpha_\ast$,
\begin{equation}\label{eq:z-tail-bd}
\alpha^{(-)}_k z e^{-c_\ast z}  \le  \ga_n(z) \le \alpha_k^{(+)} z e^{-c_\ast z} 
 \quad \forall n \ge n_k, \; \forall z \in w_k + y + [0,2 c_\ast k] .
\end{equation}
Under the event $A^{(n)}_k$, which is measurable on $\FF_k$, the latter bounds apply
for all $z=z^{(n)}_u$. Hence, we get  from \eqref{dmart.10} that 
\begin{equation*}
\begin{aligned}
& {\mathop{\underline{\lim }}_{n\to\ff}}
\bE \Big[  {\bf 1}_{A^{(n)}_k} \prod_{u \in V_k} \big(1- \alpha^{(+)}_k z_u^{(n)} e^{-c_\ast z_u^{(n)}} \big) \Big] 
 \le
{\mathop{\underline{\lim }}_{n\to\ff}} 
  \bP(t^{\ast}_n \le \tau_k(s_{n,y})\,;\, A^{(n)}_k)  \\
& \le {\mathop{\overline{\lim }}_{n\to\ff}} 
  \bP(t^{\ast}_n \le \tau_k(s_{n,y}); A^{(n)}_k) \le   
  {\mathop{\overline{\lim }}_{n\to\ff}} 
\bE \big[  {\bf 1}_{A^{(n)}_k} \prod_{u \in V_k} \big(1-\alpha^{(-)}_k z_u^{(n)} e^{-c_\ast z_u^{(n)}}\big) \big].
\end{aligned}
\end{equation*}
We now establish \eqref{ct2.17}, by taking $k \to \infty$ while utilizing \eqref{Ank-def} and \eqref{Yinf-conv}.

The same argument applies under \eqref{ct2.12}, now replacing 
$g_u$ by $g'_u$ in \eqref{eq:zv-def}, thereby changing $X_k$, $\wt X_k$ and $Y_\infty$ 
in \eqref{Xkn-conv} and \eqref{Yinf-conv}, to $X'_k$, $\wt X'_k$ and $Y'_\ff$.
\end{proof}

\section{Excursion counts to real time:\\  From \Cref{prop-eta-tau} to 
\Cref{theo-covertime}}
\label{sec-ctime}
 
\Cref{theo-covertime} amounts to showing that for any fixed $y \in \bR$ and $\ep>0$,
 \begin{eqnarray}
 \varlimsup_{n\to\ff} \bP (\CC'_{n}  \le 2 s_{n,y-2\ep})
 \le \bP(Y'_\infty  \le y) \le \varliminf_{n\to\ff} \bP(\CC'_{n} \le 2 s_{n,y+2\ep} ) \,,
\label{ct2.24}
\end{eqnarray} 
where throughout $s_{n,y}:=(m_n+y)^2/2$, as in \eqref{eq:defs}. To this end, let
\begin{equation}
R_n^s  := 2^{-n} \sum_{u \in \TT_n} T_{u,|u|}^s \,.
\label{ct2.9}
\end{equation} 
The \abbr{srw} on $\TT_n$ makes 
$2^{(n+1)} R^s_n$ steps during its first $s$ 
excursions from the root to itself. Thus,  $\{t^{\ast}_n \le \tau\} =\{\CC'_n \le  R_n^{\tau} \}$,
so for any random $t,\tau$,
\begin{equation}
\bP( t^{\ast}_ n   \le  \tau) - \bP( R_n^{\tau}   >  2 t) \le
\bP(\CC'_n  \le  2 t )\leq \bP( t^{\ast}_{n}   \le  \tau ) + \bP( R_{n}^{\tau}  < 2 t) \,.
\label{ct2.20}
\end{equation} 
Considering  \eqref{ct2.20} at $t=s_{n,y \pm 2\ep}$, the insufficient concentration of $R_n^\tau$ 
at the non-random $\tau=s_{n,y}$ rules out establishing \eqref{ct2.24} directly from \Cref{thm1.2}. 
We thus follow the approach of \cite[Section 9]{CLS}, in 
employing instead  \eqref{ct2.20} for $\tau=\tau_k(s_{n,y})$ and the $\FF_k$-measurable
\begin{equation}
\tau_k(s) := \inf \{\ell \in \bZ_+  \,|\, S_k^{\ell} \ge s\}\,, \qquad 
S_k^{\ell} := 2^{-k} \sum_{u \in V_k} T^\ell_{u,k}  \,.
 \label{ct2.11}
\end{equation}
Recall that $\bE [S_k^{1}]=1$ (see \eqref{ct2.6}), while
setting  $\bar \sigma_k^2 :=  \Var (\bar g_k) =  1 - 2^{-k}$ (see \eqref{ct10j}), 
and comparing \eqref{cov-brw} to \eqref{ct2.8}, we arrive at 
$\Var (S_k^{1}) = 2 \bar \sigma_k^2$. Hence, Donsker's invariance principle 
yields a coupling between the piece-wise linear interpolation $t \mapsto \wh S_{s,k} (t)$ 
of $\{ (S_k^t - t)/\sqrt{2s};\; t \in \bZ_+ \}$, and a 
standard Brownian motion $\{ W_\theta \}$, such that 
\begin{equation}\label{ct2.11b}
\sup_{\theta \in [0,2]} \big| \wh S_{s,k} (\theta s)
- \bar \sigma_k W_\theta \big| 
\overset{p}{\underset{s \to \ff}{\longrightarrow}} 0 \,.
\end{equation}
From \eqref{ct2.11} we see that $S_k^{\tau_k(s)} - s \ge 0$ is at most the total number of excursions from  $V_{k-1}$ to $V_k$ made by the  \abbr{srw} started at
some  $v\in V_k$, before hitting the root, plus $1$. The latter has exactly the law of 
 $2^k S_k^1$ given $S_k^1>0$. Thus, 
\begin{equation}\label{eq:ck-def}
c_k := \sup_{s \ge 0} \{ \bE [ S_k^{\tau_k(s)}  - s ] \}  \le 
\frac{\bE [ S_k^{1} ]}{\bP( S_k^1>0)} < \infty \,, 
\end{equation}
and for $\wh \tau_k(s)$ defined as in \eqref{ct2.12b},
one has when $s \to \infty$, that  
\begin{equation}\label{tauk-wlln}
 \wh S_{s,k} (\tau_k(s)) + \wh \tau_k(s) 
= \frac{S_k^{\tau_k(s)}-s}{\sqrt{2s}} 
\overset{p}{\longrightarrow} 0
\,, \qquad
\theta_s : = \frac{\tau_k(s)}{s} \overset{p}{\longrightarrow} 1 \,.
\end{equation}
In particular, considering \eqref{ct2.11b}  at  $\theta_s$, by  the continuity of 
$\theta \mapsto W_\theta$
\[
\wh S_{s,k} (\theta_s s)  
= \bar \sigma_k W_{\theta_s} +o_p(1) = \bar \sigma_k W_1 + o_p(1) =  \wh S_{s,k} (s) + o_p(1)\,.
\]
Since $\wh S_{s,k}(s)= \wh S_k(s)$ of \eqref{std-occ}, we conclude that 
$\{ \tau_k(s), s \ge 0 \}$ of \eqref{ct2.11} satisfy \eqref{ct2.12},
and with $|2 s_{n,y \pm 2 \ep} - 2 s_{n,y}|  \ge 4 \ep \sqrt{s_{n,y}}$
for  $n$ large enough,  we finish the proof of \Cref{theo-covertime} upon 
showing that for $s=s_{n,y}$ and any fixed $\ep >0$,
\begin{equation}
\lim_{k\to\ff} \varlimsup_{n \to\ff}  \bP\big(
|R_n^{\tau_k(s)} - 2 s| \ge 4 \ep \sqrt{s} \big) = 0    \,.
 \label{ct2.22}
\end{equation}
To this end,  recall that in view of \eqref{ct2.6} and \eqref{ct2.9}
 \begin{equation}
 r_j := \bE(R_j^1) = 2^{-j} |\TT_j|  = 2 - 2^{-j}  \,, 
\label{ct2.10}
\end{equation}
and similarly, by  \eqref{ct2.8} and \eqref{ct2.9} we get that 
\begin{align}
\sigma_n^2 :=  \Var(R_n^1) &=
2 \sum_{u, u' \in \TT_n} 2^{-2n} |u\wedge u'| 
\le 4 \,.
 \label{ct2.10c}
 \end{align}
Next, writing in short $\tau=\tau_k(s)$, we have for any $n \ge k$, the representation 
\[
R_n^{\tau} = 2^{k-n} R_k^{\tau} + r_{n-k} S_k^{\tau} + 2^{-k} \Delta_{k,n} (2^k S_k^{\tau}) \,,
\] 
where the random variable $\Delta_{k,n}(\ell)$ is
 the centered and scaled total time spent by the \abbr{srw} on $\TT_n$  below level $k$ during the first  $\ell$ excursions from $V_{k-1}$ to $V_k$.
 For fixed $\ell$,
$\Delta_{k,n}(\ell)
 \stackrel{(d)}{=} R_{n-k}^\ell -\bE(R_{n-k}^\ell)$ and has variance 
$\sigma_{n-k}^2 \ell \le 4 \ell$. Further, $\Delta_{k,n} (2^k S_k^{\tau})$ conditioned on  $\FF_k$ is distributed as
$R_{n-k}^\ell-\bE(R_{n-k}^\ell)$ with $\ell=2^k S_k^{\tau}$.  Recalling that $R_{k-1}^\ell = 2(R_k^\ell-S_k^\ell)$
and  utilizing \eqref{eq:ck-def}, we thus get by Markov's inequality (conditional on $\FF_k$), that  
\[
\delta_{k,n} := \bP(|R_n^{\tau} - 2^{k-n-1} R_{k-1}^{\tau} - 2 S_k^{\tau}| \ge \ep \sqrt{s} ) \le 
\frac{4 2^{-k}}{\ep^2} \frac{ \bE (S_k^\tau)}{s} = 
\frac{4 2^{-k}}{\ep^2} \big( 1 + \frac{c_k}{s} \big)   \,,
\]
goes to zero for $s = s_{n,y} \to \infty$ followed by $k \to \infty$. Also, by the union bound, 
\begin{align*}
\bP\big(
|R_n^{\tau} - 2 s| \ge 4 \ep \sqrt{s} \big)  \le \delta_{k,n} 
 + \bP(S_k^\tau - s \ge  \ep \sqrt{s} ) &+ \bP(\tau \ge 2s)
 \\&
 + \bP(2^{k-n-1} R_{k-1}^{2s} \ge \ep \sqrt{s}) \,.
\end{align*}
Next, employing Markov's inequality, we deduce that the last term goes to zero, since
$2^{k-n} \, s \, r_{k-1}/(\ep \sqrt{s}) \to 0$ for $s=s_{n,y}$ and $n \to \infty$.
By  \eqref{eq:ck-def} and having \abbr{whp}
 $\tau = \tau_k(s) < 2s$ (see  \eqref{tauk-wlln}), we thus 
arrive at \eqref{ct2.22}  and thereby conclude the proof of \Cref{theo-covertime}.

\section{Sharp right tail: auxiliary lemmas and proof of \Cref{prop-limiting-tail-gff}}\label{sec-rt}
\label{sec-limittail}
 
Hereafter we denote by 
$\bP_{s}$ probabilities of events occurring up to the completion of the first
$s$ excursions at the root and let $\eta_v (j):=\sqrt{2T^s_{v,j}}$
for $v \in V_k$, $j \le k$ and $T^s_{v,j}$
of \eqref{Tjsdef},
with the value of $s$ implicit. For $u \in V_{n'}$ where $n':=n-\ell$,
and $V^u_n := \{ v \in V_n : v(n') = u \}$, let
\begin{equation}\label{def:eta-sharp-u}
\eta^{\sharp}_\ell (u) := \min_{v\in V^u_n} \{ \eta_{v}(n) \} \,,
\end{equation}
denote the minimal (normalized) occupation time of
edges entering leaves of the sub-tree of depth $\ell$ rooted at $u$ (during the first $s$ excursions 
from the root),
abbreviating $\eta^\sharp_n$ for $\eta^\sharp_n(0)$.
Since $$\{t^{\ast}_{ n}> s\}= \{\min_{v \in V_n} \{T_{v,n}^{s}\}=0\},$$ \Cref{prop-limiting-tail-gff} 
amounts to  the claim
\begin{equation}
\label{eqforP3.1}
 \alpha_\ast =
\lim_{z\to \infty} z^{-1} \mathrm{e}^{c_\ast z}
{\mathop{\underline{\overline{\lim }}}_{n\to\ff}}
 \bP_{s_{n,z}}(\eta^{\sharp}_n =0),
\end{equation}
for $s_{n,z}$ and $c_\ast$ of \eqref{eq:defs} and
\eqref{28.00}, respectively.
Our proof of \eqref{eqforP3.1} is based on a refinement of the probability estimates 
of \cite[Section 5]{BRZ}, 
intersecting here the event $\{\eta^\sharp_n = 0\}$ with barrier events involving 
the (normalized) edge occupation times $\{ j \mapsto \eta_v(j), v \in \TT_n \}$.  
More precisely, we adapt the strategy of \cite[Section 3]{BDZconvlaw}, by
essentially bounding $\bP_{s_{n,z}} (\eta^\sharp_n=0)$ between the expectations
of counts $\Lambda_{n,\ell} \le \Gamma_{n,\ell}$ for two barrier type events,
which are equivalent at the claimed scale of asymptotic growth in $z$ 
(see \Cref{lem-Gamma-Lambda}). Our curved barrier event for $\Gamma_{n,\ell}$ 
is relaxed enough to deduce that the event $\{\Gamma_{n,\ell} \ge 1\}$ is for large $n$, $\ell$, 
about the same as having $\{\eta_n^\sharp=0\}$ (see \Cref{lem-G-neglig}). 
The straight barrier event for $\Lambda_{n,\ell}$ is strict enough to yield a 
negligible variance (see \Cref{lem-second-moment}), so its expectation serves
to lower bound $\bP_{s_{n,z}} (\eta^\sharp_n=0)$. Our claim 
\eqref{eqforP3.1} then follows from such a limit for $\bE_{s_{n,z}} [\Lambda_{n,\ell} ]$
(which is a consequence of \Cref{prop-asymptotic-first-moment}). Specifically,
for $s=s_{n,z}$  consider the excess
edge occupation times, over the 
barrier
\begin{equation}\label{eq:bar-st}
\bar \varphi_n(j) := \rho_n (n-j) \,, \quad \quad  j \in [0,n'] \,, \quad\quad n'=n-\ell\,.
\end{equation}
 In the sequel we show that the 
main contribution to $\{\eta_n^\sharp=0\}$ is due to not covering a sub-tree 
rooted at some $u \in V_{n'}$
while the edge occupation times along the geodesic to $u$  
exceed the barrier $\bar \varphi_n(\cdot)$ of \eqref{eq:bar-st}, 
with the excess  at the edge into $u$ further restricted to 
\begin{equation}\label{dfn:I_ell}
I_\ell:=\sqrt{\ell} \, [r^{-1}_\ell,r_\ell],   \qquad   r_\ell:=\sqrt{\log \ell}\,.
\end{equation}
To this end, let 
\begin{equation}\label{def:eta-hat}
\hat \eta_v(j) := \eta_v(j) - \bar \varphi_n(j) \,,
\end{equation}
considering for
$u \in V_{n'}$ the events
\begin{align}\label{eq-big-definition-sharp}
E_{n,\ell}(u) &: = 
 \bigcap_{0 \le j \le n'} \{\hat \eta_u(j) > 0 \} \bigcap \{\hat \eta_u (n')  \in  I_\ell,\, 
 \eta^\sharp_\ell (u) = 0 \}  \,,
\end{align}
and the corresponding counts
\begin{equation}\label{def:Lambda-n}
\Lambda_{n,\ell} := \sum_{u \in V_{n'}}{\bf 1}_{E_{n,\ell}(u)}\,.
\end{equation}
\begin{figure}
\begin{center}
\begin{tikzpicture}[scale=0.6]

\begin{scope}[shift={(-5,0)}]
    \draw [solid,->] (0.0,0) -- (14.4,0) node [ right]  {level};
    \draw [solid,->] (0,0) -- (0,9.2) node [below left] {};
  \draw [dotted] (13,0) -- (13,9);
    \draw[-] (0,7.8) .. controls (6,3) .. (13,0.3);
\draw[dashed] (0,8.7) .. controls (2,6.5) and  (4,7.1) .. (6,5.1);
\draw[dashed] (6,5.1) .. controls (8,4) and  (10,5) .. (13,1);
\draw[dotted,thick] (0,8.7) .. controls (2,5) and  (4,6.9) .. (6,5);
\draw[dotted,thick] (6,5) .. controls (8,3.2) and  (10,3) .. (13,1);
\draw[red] (13,1) ..controls (13.5,1.1) .. (14,1.4);
\draw[red] (13,1) ..controls (13.5,1.1) .. (14,0.8);
\draw[red] (13,1) ..controls (13.6,1.1) .. (14,0.4);
\draw[red] (13,1) ..controls (13.5,0.8) .. (14,1.3);
\draw[red] (13,1) ..controls (13.5,0.8) .. (14,0.3);
\draw[red] (13,1) ..controls (13.5,0.8) .. (14,0.7);
\draw[red,thick] (13,1) ..controls (13.4,0.4) .. (13.8,0);
\draw[red] (13,1) ..controls (13.4,0.6) .. (14,0.2);
 \draw[-{>[length=7,width=10]}, thick] (11,7) -- (13,1);
  \end{scope}
   \begin{scope}[shift={(-5,0)}]
\draw[-] (0,8) -- (14,0);
 \node at (14,-0.3) {\small $n$};
 \node at (13,-0.3) {\small $n'$};
 \node at (-0.5,8.1) {\small $m_n$};
  \node at (-1.1,7.7) {\small $m_n-h_\ell$};
   \node at (-1.1,8.7) {\small $m_n+z$};
   \node at (11,7.2) {$\rho_n \ell+r$};  
  \end{scope}


\end{tikzpicture}
\end{center}
\caption{Depiction of the events $E_{n,\ell}(u)$ (dashed line) and $F_{n,\ell}(u)$ (dotted line) for some 
$u\in V_{n'}$. In either case, the red paths emanating from
 level $n'=n-\ell$ denote excursion counts corresponding to different children of $u$. Note the curved vs. straight barrier and the excursion count that reaches $0$.}
\label{fig:EnlandFnl}
\end{figure}
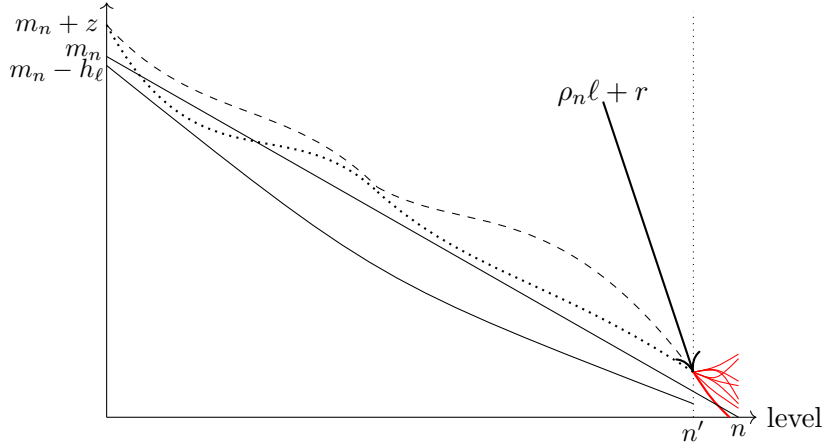
See Figure \ref{fig:EnlandFnl} for a pictorial illustration of the event $E_{n,\ell}(u)$.
As explained before, aiming first to upper bound 
$\bP_{s_{n,z}}(\eta^\sharp_n =0)$, we fix $\delta \in (0,\frac12)$
and for $k \in [1,n]$, $h \in [0,n-k]$, consider the curved, relaxed barriers
\begin{equation}\label{eq:barrier}
\varphi_{n,k,h} (j):=\bar \varphi_n (j) - \psi_{k,h}(j) \,, \quad  j \in [0,k] \,,  
\end{equation}
using hereafter  the notations
\begin{equation}\label{psi-def}
\psi_{k,h}(j):= h + j_k^\delta\,, \quad j_k:= j \wedge (k-j), \qquad j \in [0,k] \,.
\end{equation}
We further use the abbreviated notation 
\begin{equation}\label{dfn:psi-ell}
\psi_{\ell}(\cdot):=\psi_{n',h_\ell}(\cdot),   \quad \mbox{where} \quad h_\ell := \frac{1}{2} \log \ell \,,
\end{equation}
with $n'=n-\ell \ge 1$. 
Replacing the barriers of \eqref{eq:bar-st} by those of
\eqref{eq:barrier}, we then form the larger counts 
\begin{equation}
\Gamma_{n,\ell} = \sum_{u \in V_{n'}}
{\bf 1}_{F_{n,\ell}(u)}\,,
\end{equation}
where in terms of \eqref{def:eta-sharp-u},  \eqref{def:eta-hat} and \eqref{eq:barrier}, 
we define for each $u \in V_{n'}$ 
\begin{align}\label{eq-big-definition}
F_{n,\ell}(u)  &: = \bigcap_{0 \le j \le n'}
\{\hat \eta_{u}(j) + \psi_{\ell}(j) > 0  \}
\bigcap \{ \eta^\sharp_\ell (u) = 0 \} 
\,.
\end{align}
See Figure \ref{fig:EnlandFnl} for a pictorial  illustration of the event $F_{n,\ell}(u)$.
 If $\eta^\sharp_n=0$, then necessarily $\eta^\sharp_\ell(u)=0$ for some $u \in V_{n'}$
and either $F_{n,\ell}(u)$ occurs (so $\Gamma_{n,\ell} \ge 1$), or else the event 
$G_{n,\ell} := G_{n,n'}(h_\ell)$ must occur, where, see   Figure \ref{fig:Gnl},
\begin{align}\label{eq-def-G-N-prelim}
G_{n,k'} (h) &:= \bigcup_{u\in V_{k'}} \bigcup_{0 \leq j \leq k'}\{ \, \hat \eta_{u}(j) \le  -\psi_{k',h} (j)  \} \,.
\end{align}
Hence, for any $\ell$,
\begin{equation}\label{eq:ubd-by-Gamma}
\bE_{s_{n,z}} [ \Gamma_{n,\ell} ] \ge 
 \bP_{s_{n,z}} (\Gamma_{n,\ell} \geq 1)  \ge 
\bP_{s_{n,z}}(\eta^{\sharp}_n =0) - 
\bP_{s_{n,z}}(G_{n,\ell}) \,.
\end{equation}
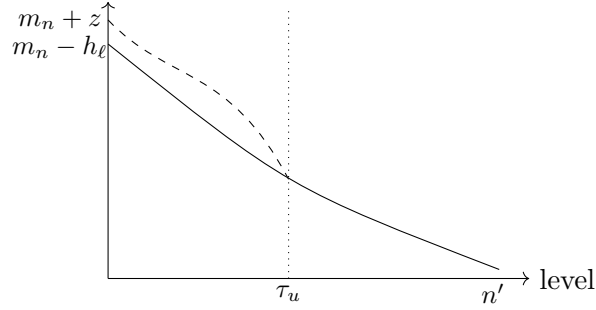
\begin{figure}
\begin{center}
\begin{tikzpicture}[scale=0.4]

\begin{scope}[shift={(-5,0)}]
    \draw [solid,->] (0.0,0) -- (14.0,0) node [ right]  {level};
    \draw [solid,->] (0,0) -- (0,9.2) node [below left] {};
  \draw [dotted] (6,0) -- (6,9);
    \draw[-] (0,7.8) .. controls (6,3) .. (13,0.3);
\draw[dashed] (0,8.6) .. controls (2,6.5) and  (4,7.1) .. (6,3.3);
  \end{scope}
   \begin{scope}[shift={(-5,0)}]
 \node at (6,-0.5) {\small $\tau_u$};
 \node at (12.8,-0.5) {\small $n'$};
  \node at (-1.6,7.7) {\small $m_n-h_\ell$};
   \node at (-1.6,8.6) {\small $m_n+z$};
  \end{scope}


\end{tikzpicture}
\end{center}
\caption{
 Depiction of an event from $G_{n,\ell}$ (dashed line) corresponding to some 
 $u\in V_{n'}$, $n'=n-\ell$. Note the curved barrier.
}
\label{fig:Gnl}
\end{figure}
Recall that by \cite[proof of Corollary 5.4]{BRZ}, 
for some $c'>0$ and all $z \ge 1$,
\begin{equation}\label{eq2.7}
\varliminf_{n \to \infty} \bP_{s_{n,z}}(\eta^{\sharp}_n =0) \ge c' z \mathrm{e}^{-c_\ast z} \,,
\end{equation}
so our next lemma, which is an immediate consequence of  \Cref{lem-a-priori2} below,
shows that the 
right-most term in \eqref{eq:ubd-by-Gamma} is negligible.
\begin{lemma}\label{lem-G-neglig}
We have that 
\begin{equation}\label{gee-final}
\lim_{\ell \to \infty} \sup_{z \ge 1} \{ z^{-1} e^{c_\ast z} \varlimsup_{n \to \infty} \bP_{s_{n,z}} (G_{n,\ell}) \} = 0 \,. 
\end{equation} 
\end{lemma}
Combining \eqref{eq:ubd-by-Gamma}-\eqref{gee-final}, we arrive at
\begin{equation}
\label{eq3.60-jay-nn}
\varlimsup_{\ell \to \ff}
\varlimsup_{z\to \infty} \varlimsup_{n\to \infty}
\frac{\bP_{s_{n,z}}(\eta^{\sharp}_n =0)}
{\bE_{s_{n,z}}[\Gamma_{n,\ell}]}\le 1 \,.
\end{equation}
Restricting hereafter to $\delta \in (0,\frac{1}{6})$ 
allows us to further show in  \Cref{zerosubsection} the following equivalence of  first moments
(c.f. \eqref{delta-under-16} for why we take $\delta$ small).
\begin{lemma}\label{lem-Gamma-Lambda}
For any $\delta \in (0,\frac{1}{6})$ we have that 
\begin{equation}
  \label{eq-clear240113}
\lim_{\ell \to \ff} \varlimsup_{z\to \infty} \{
z^{-1} e^{c_\ast z} \varlimsup_{n\to \infty} \bE_{s_{n,z}} [\Gamma_{n,\ell} - 
\Lambda_{n,\ell}] 
\} = 0 \,.
\end{equation}
\end{lemma}
\noindent
Now, from  \eqref{eq3.60-jay-nn} and \eqref{eq-clear240113}, we have  the upper bound
\begin{equation}
\label{eq3.60nn}
\varlimsup_{\ell \to \ff}
\varlimsup_{z\to \infty} \varlimsup_{n\to \infty}
\frac{\bP_{s_{n,z}}(\eta^{\sharp}_n =0)}
{\bE_{s_{n,z}}[\Lambda_{n,\ell}]}\le 1 \,.
\end{equation}
\noindent 
For such expected counts with straight barriers,  we establish in \Cref{sec-3.5}, 
using the connection to the $0$-Bessel process, 
the following large $n$ and $z$ asymptotic.
\begin{proposition}\label{prop-asymptotic-first-moment}
There exists $\alpha_\ell >0$ such that 
\begin{align}\label{eqmainresultinsecondsection}
\lim_{\ell \to \infty}  \alpha_{\ell}^{-1} \big\{
{\mathop{\underline{\overline{\lim }}}_{z \to\ff}}
z^{-1} \mathrm{e}^{c_\ast z} 
{\mathop{\underline{\overline{\lim }}}_{n\to\ff}}
\bE_{s_{n,z}}[\Lambda_{n,\ell}] \big\} = 1 
 \,,
\end{align}
where by \eqref{eq2.7} and \eqref{eq3.60nn}, 
$\liminf \{\alpha_{\ell}\}$ is strictly positive.
\end{proposition}
\noindent
As shown in \Cref{firstsubsection},
the barrier event we have added in the definition \eqref{eq-big-definition-sharp} of $E_{n,\ell}(u)$ 
yields the following tight control on the second moment of $\Lambda_{n,\ell}$.
\begin{lemma}\label{lem-second-moment}
We have that 
\begin{equation}
\label{newequation72}
\lim_{\ell \to \ff}
\varlimsup_{z\to \infty} \big\{ z^{-1} e^{c_\ast z} \varlimsup_{n \to \infty} 
\bE_{s_{n,z}} [\Lambda_{n,\ell} (\Lambda_{n,\ell}-1) ]  \big\} = 0\,.
\end{equation}
\end{lemma}
\noindent
Note that $\Lambda_{n,\ell} \ge 1$ implies having $\eta_v(n)=0$ for some $v \in V_n$,
that is, having $\eta^\sharp_n=0$. Hence, with $\Lambda_{n,\ell}$ integer valued, 
for any choice of $\ell$,
\begin{align}\label{eq:2mom-lbd}
\bP_{s_{n,z}}(\eta^{\sharp}_n =0) \geq \bP_{s_{n,z}}(\Lambda_{n,\ell} \ge 1)
\geq \bE_{s_{n,z}} [\Lambda_{n,\ell}] - \bE_{s_{n,z}}[\Lambda_{n,\ell} (\Lambda_{n,\ell}-1)] \,.
\end{align}
Having a positive $\liminf \{\alpha_\ell \}$,  the latter bound, together 
with \eqref{eqmainresultinsecondsection} and  \eqref{newequation72}, imply that 
\begin{equation}
\label{eq3.61nn}
\varliminf_{\ell \to \infty} \varliminf_{z\to \infty}  \varliminf_{n\to\ff}
\frac{\bP_{s_{n,z}}(\eta^{\sharp}_n =0)}
{\bE_{s_{n,z}}[\Lambda_{n,\ell}]}\ge 1 \,.
\end{equation}
\begin{proof}[Proof of \Cref{prop-limiting-tail-gff}] Comparing first \eqref{eq3.60nn} to \eqref{eq3.61nn} and then
with \eqref{eqmainresultinsecondsection}, we  conclude that
\begin{equation}\label{eq:punch-line}
\lim_{\ell \to \infty}  \alpha_{\ell}^{-1} \big\{
{\mathop{\underline{\overline{\lim }}}_{n\to\ff}}
 z^{-1} \mathrm{e}^{c_\ast z} 
{\mathop{\underline{\overline{\lim }}}_{n\to\ff}}
\bP_{s_{n,z}} (\eta^\sharp_n = 0) \big\} = 1 \,.
\end{equation}
Necessarily $\alpha_\ell \to \alpha_\ast$ for which \eqref{eqforP3.1} holds
(with $\alpha_\ast>0$  in view of \eqref{eq2.7} and $\alpha_\ast<\infty$ by 
\cite[Proposition 5.2]{BRZ}).
\end{proof}


\section{Barrier bounds for excursion counts}\label{sec-re}

We keep the barrier sequences of \eqref{eq:barrier} and 
all other related notation from \Cref{sec-rt}. Further, with $\rho_n \to c_\ast > 1.1$, see \eqref{28.00}, \abbr{wlog}
we restrict to $n \ge n_\ast \ge 64$ with $\rho_n \ge \rho_\ast =: 1.1$, starting at 
the following a-priori bound on the events $G_{n,k'}(h)$ from \eqref{eq-def-G-N-prelim}. 
Recall the notation $s_{n,z}$ of 
\eqref{eq:defs}.
\begin{lemma}\label{lem-a-priori2} 
For some $c<\infty$, any $n \ge n_\ast$, $z,k' \ge 1$ and $h \in [0,n-k']$,
\begin{equation}
\bP_{s_{n,z}} (G_{n,k'}(h))  \le c (z+h) e^{ -c_\ast (z+h)} e^{-(z+h)^2/(8n)} \,.
\label{gee.2}
\end{equation}  
\end{lemma}
\begin{proof}[Proof of \Cref{lem-G-neglig}]
Setting $h=h_\ell=
 \frac{1}{2} \log \ell$ in \eqref{gee.2}, results with 
\begin{align}
\bP_{s_{n,z}}\big(G_{n,\ell}\big)  &\le 
c(z+\log \ell) \ell^{-c_\ast/2}  e^{ -c_\ast z}e^{-z^2/(8n)}\,,
\label{gee.3}
\end{align}	
so taking $\ell \to \infty$ 
establishes \eqref{gee-final}.
\end{proof}
Before embarking on the proof of \Cref{lem-a-priori2}, we deduce from it 
certain useful a-priori tail bounds on the non-covering events $\{\eta^\sharp_\ell = 0 \}$.
\begin{corollary}
\label{lem-prelim-taillower}\label{lem2.7}\label{cor: Upper Bound cover} 
For some $c<\infty$ and  all  $n \ge n_\star$, $z\geq 1$,
\begin{equation}\label{eq-right-tail-ex}
\bP_{s_{n,z}}(\eta^{\sharp}_n =0) \le  c z e^{ -c_\ast z} e^{-z^2/(8n)}\,.
\end{equation}
Further, for some $\hat \ell$ finite, 
any $\hat \ell \le \ell \leq n/\log n$ and $r \ge - h_\ell$,
\begin{equation}
\ga_{n,\ell}(r) := \bP_{[(\rho_{n}\ell+r)^{2}/2]}(\eta^{\sharp}_{\ell} =0) \le 
c \ell^{-1} (r+\log \ell) e^{-c_\ast r} e^{-r^{2}/(8 \ell)} \,.
		\label{35.1}
\end{equation}
\end{corollary}
\begin{proof} The event $\eta^\sharp_n =0$ amounts to $\eta_v(n)=0$ for some $v \in V_n$.
With $\varphi_{n,n,0}(n)=0$ (see \eqref{eq:barrier}-\eqref{psi-def}), this implies that 
$\eta_v(n) \le \varphi_{n,n,0}(n)$ and consequently that 
$G_{n,n}(0)$ also occurs (take $j=k'=n$ 
in \eqref{eq-def-G-N-prelim}). 
That is $\{\eta^\sharp_n = 0\} \subseteq G_{n,n}(0)$, so
the bound \eqref{eq-right-tail-ex} follows from \eqref{gee.2}.
Proceeding to prove \eqref{35.1}, setting $\hat \ell := n_\star \vee \exp(2(c_\ast + 1)/(2-c_\ast))$, 
one easily
checks that $z = r + (\log \ell - 1)/c_\ast \ge 1$ whenever 
$r \ge - \frac{1}{2} \log \ell$ and $\ell \ge \hat \ell$. If further
$1 \le \ell \le n/\log n$, then 
\[
\rho_{n} \ell + r =  c_\ast \ell - \frac{\ell \log n}{c_\ast n} + r 
\ge m_\ell + z \,,
\] 
so \eqref{eq-right-tail-ex} at $n=\ell$ and such $z$ yields the bound \eqref{35.1},
possibly with $c \mapsto e c$.
\end{proof}
\noindent
We recall \eqref{def:eta-hat},  \eqref{psi-def} and take throughout 
\begin{equation}\label{hx-def}
H_x:=[x,x+1] \,.
\end{equation}
The key to 
this section are
the following a-priori barrier  estimates adapted from  \cite[Section 5]{BRZ} (though it is 
advised to skip the proofs at first reading).
\begin{lemma}\label{lem:barrier-est}
Let  $g_m(i):=i \exp(c_\ast i + i^2/m)$, $\beta_{n,k} := \frac{k}{n} \log n - \log k$
and $z_h:=z+h$. For some $c_1 \ge 1$, all $z \ge 1$, $k \ge 0$, $h \in [0,n-k)$ and $i \in \bZ_+$
\begin{align}\label{eq-oferlau1-def}
q_{n,k,z} & (i;h) 
:= \bP_{s_{n,z}} (\min_{j \le k} \{ \hat \eta_v(j) + \psi_{k,h}(j)  \} \ge 0,  \; \hat \eta_v(k) \in H_{i-h})  \\
\le &  c_1 2^{-k} \big( \frac{z_h}{\sqrt{k_n}} \wedge k \big) \, e^{\beta_{n,k}} \, e^{- c_\ast z_h} 
e^{-z_h^2/(4m)} \, g_m (i+1) \,, \;\;  \forall m \in [2k,n^2] 
\label{eq-oferlau1}
\end{align}
(replacing for $k=0$ the ill-defined factor $\big(\frac{z_h}{\sqrt{k_n}} \wedge k\big) e^{\beta_{n,k}}$ by $1$).

Likewise, for $i, k'  \in \bZ_+$, $n' = k'+k \in (k',n)$, $m \in [2k,(n-k')^2]$ and $z \ge 0$,
 \begin{align}\label{eq-amirlau1}
  p_{n,k,z} (i)  &:=   
\bP(\min_{j \in (k',n']} \{ \hat \eta_{v}(j) \}  \ge  0, \hat \eta_v (n')  \in  H_i \,|\, \hat \eta_v(k') = z) \nn \\ 
& \le c_1 
 \frac{2^{-k} e^{\beta_{n,k}}}{\sqrt{k_{n-k'}}} (z \vee 1) e^{- c_\ast z} 
e^{-(z \vee 1)^2/(4m)} \, g_m (i+1) \,.
\end{align}
The bound \eqref{eq-amirlau1} applies also to  $z \in [- \rho_n k,0]$, now with $m=-4k$. 
\end{lemma}
\begin{proof} In case $k \ge 1$, setting $a = \rho_n n - h
$ and $b = \rho_n (n-k) - h
$,  the event considered in \eqref{eq-oferlau1-def} corresponds to  \cite[(1.1)]{BRZ}
for $L=k$, $C=1$, $\vep=\frac{1}{2}-\delta$ and the line $f_{a,b}(j;k)$ between $(0,a)$ and $(k,b)$,
taking there 
$y=b+i$ and $x= m_n+z \ge a$.  Having $z \ge 1$ and $m_n \ge 1$
yields that $x \ge 2$. Further, with $h \in [0,n-k)$ and  $\rho_n \ge \rho_\ast$,  
\begin{equation}\label{eq:y-lbd}
y \ge b  = \rho_n  (n-k) - h  \ge  1 + \frac{1}{10} (n-k)  
\end{equation}
(the restriction to $y \ge \sqrt{2}$ in \cite[(1.1)]{BRZ} clearly can be replaced with $y> 1$ there, since for $y\in [1,\sqrt{2})$ 
one has $H_y^2/2 \cap \bZ\subset (H_{\sqrt{2}})^2/2 \cap \bZ$). 
Here $x/L$ and $y/L$ are not uniformly bounded above but
following \cite[proof of (4.2)]{BRZ}
and utilizing \cite[Remark 2.6]{BRZ} to suitably modify \cite[(4.16)]{BRZ},
we nevertheless arrive at the bound 
\begin{align}\label{eq-oferlau}
q_{n,k,z}(i;h) \le c \, 
\frac{(1+z_h)(1+i)}{k}
\sqrt{\frac{x}{ky}} \sup_{w \in H_y} \{ e^{-(x-w)^2/(2k)} \} \,.
\end{align}
In addition, whenever $x \ge \sqrt{2}$, $y,k \ge 1$, we have that 
\begin{align}\label{eq-oferlau-s}
q_{n,k,z}(i;h) \le \bP_{s_{n,z}} ( \hat \eta_v(k) \in H_{i-h}) \le c \sup_{w \in H_y} \{ e^{-(x-w)^2/(2k)} \} 
\end{align}
(see \cite[Lemma 3.6]{BRZ17}).
Next, since $x \le c_\ast n + z$,  we deduce from \eqref{eq:y-lbd} that 
\begin{equation}\label{eq:bd-xky}
\frac{x}{ky} \le \frac{20}{k_n} \Big(c_\ast + \frac{z}{n} \Big) \le \frac{c_o}{k_n} \exp\big(\frac{z^2}{12 n^2}\big) \,,
\end{equation}
for some 
constant $c_o<\infty$.
Further, as $x-a=z_h$ we have from \eqref{28.00} that 
\begin{equation}
\label{eq-latenight}
x - b = c_\ast k + z_h  - \vep_{n,k} \,, \qquad  \vep_{n,k} := \frac{k \log n}{c_\ast n} \,,
\end{equation}
and since $\frac{c^2_\ast}{2} = \log 2$, we get,  for any real  $\wt w$,
\begin{equation}\label{eq:jay-ident1}
\frac{1}{2k} (x-b-\wt w)^2 = k \log 2 + c_\ast (z_h-\vep_{n,k}-\wt w) + \frac{1}{2k} (z_h-\vep_{n,k}-\wt w)^2 \,.
\end{equation}
By an elementary inequality, for any $m \ge 2k$,
\begin{equation}\label{eq:bd-el-xw}
(z_h - \wt w-\vep_{n,k})^2 \ge \frac{2k}{3m} (z_h^2 - 3 \wt w^2 - 3 \vep_{n,k}^2) \,, 
\end{equation}
so that by \eqref{eq:jay-ident1}
\begin{equation}\label{eq:bd-xbw2}
\frac{1}{2k} (x-b-\wt w)^2  \ge 
 k \log 2 + c_\ast \big(z_h - \vep_{n,k}  - \wt w \big) 
+ \frac{z_h^2}{3m} - \frac{\wt w^2}{m} - \frac{\vep_{n,k}^2}{m}  \,.
\end{equation}
With $c_\ast \vep_{n,k} - \log k = \beta_{n,k} \le 1$ for $k \in [1,n]$ 
(and $\vep_{n,k}^2/k$ uniformly bounded), 
we plug into the smaller among 
\eqref{eq-oferlau} and \eqref{eq-oferlau-s}
the bounds  \eqref{eq:bd-xky} and \eqref{eq:bd-xbw2} 
(for $w=b+\wt w$ and $\wt w \in H_i$),
to arrive at \eqref{eq-oferlau1}.
By definition $q_{n,0,z}(i;h) = {\bf 1}_{z_h}(i)$,  and since
$g_m(z_h+1) \ge \exp(c_\ast z_h + z_h^2/(4m))$,
clearly \eqref{eq-oferlau1} holds also for $k=0$
(under our convention).

Turning to the proof of  \eqref{eq-amirlau1}, we consider first $z > 0$, proceeding as in the proof 
of \eqref{eq-oferlau1} with the line $f_{a,b}(j;k)$ of length 
$k$ and same slope as in the preceding, now connecting $(k',a)$ to $(n',b)$,
where $a =
\rho_n (n-k')$ and $b = 
\rho_n (n-n')$. For $x=a+z$ and $y=b+i$
we have $x \ge a > b$ and $y \ge b>0$, thanks to our assumption that $n' \in (k',n)$.
By the Markov property of $j \mapsto \eta_v(j)$, for such values of $(a,b,x,y)$ 
the \abbr{rhs} of \eqref{eq-oferlau} with $h=0$ necessarily bounds the probability $p_{n,k,z}(i)$,
and thereafter one merely follows  the derivation of \eqref{eq-oferlau1}, now with  
$h=0$ and $n-k'$ replacing $n$ when bounding $x/y$. The latter modification 
results in having $\sqrt{k_{n-k'}}$ in \eqref{eq-amirlau1}, instead of $\sqrt{k_n}$. 
Next, for $z \le 0$ we simply lower the barrier line $f_{a,b}(j;k)$ to start at
$a=x=\bar \varphi_n(k') + z$, where our assumption that $z \ge -\rho_n k$ guarantees
having $x \ge \bar \varphi_n(n') >0$ (thereby yielding \eqref{eq-oferlau}). Here
$x/(ky) \le c_o/k_{n-k'}$, and $z_h =z \le 0$ allows us to replace the
\abbr{rhs} of \eqref{eq:bd-el-xw} by $\frac{\wt w^2}{2} -  \vep_{n,k}^2$, 
yielding the stated form of \eqref{eq-amirlau1}.
\end{proof}

We conclude this sub-section by adapting the bounds of \Cref{lem:barrier-est} to the form 
needed when proving \Cref{lem-second-moment} and \Cref{lem-Gamma-Lambda}. 
\begin{lemma}\label{lem-apriori-c}
There exists a constant $c_2<\infty$ satisfying the following.
Fix $\hat \ell$ as in   \Cref{lem2.7} and $I_\ell$ as in  \eqref{dfn:I_ell}. 
For  any $z \ge 0$,  $\ell \in [\hat \ell, \frac{n}{\log n}]$ and  $k\geq 1$, setting $k'=n'-k$ with $n'$ as in \eqref{eq:bar-st},
\begin{align}\label{eq:bd-theta}
\theta_{n,k,\ell} (z) &:=   
\bP
(\min_{j \in (k',n']} \{ \hat \eta_{v}(j) \}  \ge 0, \, \hat \eta_v (n')  \in  I_\ell,\
\eta^\sharp_\ell(v) = 0 \,|\, \hat \eta_v(k') = z)  \nn \\
& \le c_2 \,
2^{-k} e^{\beta_{n,k}} (1 \vee \sqrt{\ell/k})  (1+z) e^{- c_\ast z} e^{-z^2/(8(k \vee 8\ell))} 
\end{align}

If in addition  $z \ge 4 h_\ell$, $n \ge 3 \ell$, then for any $r \ge 0$, 
\begin{align}\label{eq:bd-phi}
{\bf p}_{n,k,z}(r) &:=  \bP_{s_{n,z}}(\min_{j \le n'} 
\{\hat \eta_v (j) + \psi_{\ell}(j) {\bf 1}_{j \le k}  \} > 0 \,, \; \hat \eta_{v} (k)  \le 0, \, \; \hat \eta_{v} (n')  \in  H_r ) 
 \nn \\
&\le c_2 \, 2^{-n'} 
\frac{\psi_\ell(k)^3}{k_{n'}^{3/2} \ell^{1/2}} 
e^{-z^2/(16k)} \,  z  (r+1) e^{-c_\ast (z- r)} e^{-\frac{r^2}{4(n'-k)}}\,.
\end{align}
\end{lemma}
\begin{proof} Starting with \eqref{eq:bd-theta}, by the Markov property of $\eta_v(j)$ at $j=n'$
and monotonicity of $y \mapsto \ga_{n,\ell}(y)$ of \eqref{35.1}, we have 
in terms of $p_{n,k,z}(\cdot)$ of \eqref{eq-amirlau1}
\[
\theta_{n,k,\ell}(z) \le  \sum_{r \in I_\ell} p_{n,k,z}(r) \ga_{n,\ell}(r) \,.
\]
Plugging the bounds of \eqref{35.1} and \eqref{eq-amirlau1} (at $m=2 (k  \vee 8 \ell)$), 
yields for $c'_1$ finite
\[
\theta_{n,k,\ell}(z) \le c'_1
 \frac{2^{-k} e^{\beta_{n,k}}\sqrt{\ell}}{\sqrt{k_{n-k'}}} (1+z) e^{- c_\ast z} 
e^{-z^2/(8(k \vee 8 \ell))} \, \sum_{r \in I_\ell} \frac{r^2}{\ell^{3/2}}   e^{-r^{2}/(16 \ell)} \,.
\]
With the latter sum uniformly bounded and $k_{n-k'}=k \wedge \ell$, we arrive at \eqref{eq:bd-theta}.

Next, turning to establish \eqref{eq:bd-phi}, note first that  
\begin{equation}\label{eq:psi-pre-dom}
j_{k'}^\bb\leq k_{k'}^\bb+j_k^\bb \,, \qquad\qquad \forall \; 0 \le j \le k < k'
\end{equation}
when $\bb=1$ and consequently also for all $\bb \in [0,1]$.  For $\bb=\delta$, it results with
\[
\psi_{k',h}(j) \le \psi_{k,h'}(j) \,, \quad \mbox{ for } \quad h'=\psi_{k',h}(k), \quad j \in [0,k] ,
\]
with equality at $j=k$. In particular, recalling  \eqref{dfn:psi-ell} and considering 
$k'=n'$, we have  that 
\begin{equation}\label{eq:psi-dom}
\psi_\ell(j) \le \psi_{k,h}(j) \quad \mbox{ for } \quad h = \psi_\ell(k), \;\;  j \in [0,k] \,.
\end{equation}
 Employing \eqref{eq:psi-dom} to enlarge 
the event whose probability is ${\bf p}_{n,k,z}(i)$, we get  by the Markov property of
$\eta_v(j)$ at $j=k$, in terms of $q_{n,k,z}(\cdot;\cdot)$, 
$p_{n,k,z}(\cdot)$
and $h=\psi_\ell(k)$, that 
\begin{align}\label{eq:varphi-def}
{\bf p}_{n,k,z}(r) \le &  \sum_{i=1}^{h} q_{n,k,z} (h-i;h) 
\sup_{z' \in H_{-i}} \{ p_{n,k',z'}(r) \} 
\,.
\end{align}
Substituting first our bound \eqref{eq-amirlau1} at $m=-4k'$, and then 
\eqref{eq-oferlau1} at $m=2k$, yields that for 
some $c'_1
$ finite,
\begin{align}\label{eq:phi-bd2}
& {\bf p}_{n,k,z}(r) \le  c_1 \sum_{i=1}^{h} q_{n,k,z} (h-i;h) 
2^{-k'} e^{\beta_{n,k'}} (k'_{k'+\ell})^{-1/2}  e^{c_\ast  (r+i)} (r+1) e^{-\frac{r^2}{4k'}} \nn \\
&\le  c'_1 \, h \, 2^{-n'}
\frac{e^{\beta_{n,k} + \beta_{n,k'}}} {\sqrt{k'_{k'+\ell} k_n}} z_h e^{-c_\ast z_h} 
 e^{-z^2/(8k)}  g_{2k}(h) (r+1) e^{c_\ast r}  e^{-r^2/(4k')}\,.
\end{align}
With $\delta \le \frac{1}{2}$ and $z \ge 4 h_\ell$, it follows that 
\[
\frac{h^2}{2k} \le \frac{h_\ell^2 + k_{n'}^{2\delta}}{k} 
\le \frac{z^2}{16k} + 1 \,.
\]
Further, recall that $k+k'=n'$ and $\beta_{n,n'} \le 1$, hence 
\begin{equation}\label{eq:sum-beta}
e^{\beta_{n,k}+\beta_{n,k'}} = e^{\beta_{n,n'}} \frac{n'}{k k'} \le \frac{2 e}{k_{n'}} \,.
\end{equation}
Our assumption $n \ge 3 \ell$ results with $k \vee k' \ge \ell$ and thereby 
$k'_{k'+\ell} \, k_n \ge \ell \, k_{n'}$. Applying the preceding 
within \eqref{eq:phi-bd2}, we arrive at \eqref{eq:bd-phi} .
\end{proof}

\subsection{Negligible crossings: Proof of \Cref{lem-a-priori2}}
Fixing $n,k',h$ as in \Cref{lem-a-priori2}, consider for $u \in V_{k'}$,  the 
first time 
\[
\tau_u := \min\{ j \ge 0 : \eta_u (j) \le \varphi_{n,k',h} (j) \} 
\]
that the process $j \mapsto \eta_u (j)$ reaches 
the relevant barrier of \eqref{eq:barrier}, see Figure \ref{fig:Gnl}. For $z \ge 1$ we have that $\tau_u \ge 1$, since $\eta_{u}(0)=\sqrt{2s_{n,z}} = m_n + z > \varphi_{n,k',h}(0)$
under $\bP_{s_{n,z}}$.
Decomposing $G_{n,k'}(h)$ according to the possible values of $\{\tau_u\}$, results with
		\begin{align}
	\bP_{s_{n,z}}(G_{n,k'}(h))		
		&\le \sum_{k=1}^{k'}
		\underset{({\rm I}_k)}{\underbrace{
			\bP_{s_{n,z}}
				\Big(\exists u\in V_{k'} \; \text{ such that } \;\tau_{u} = k \Big) 
				}}				\,.
				\label{eq: upper bound split}
		\end{align}
The event $\{\tau_u=k\}$ depends only on the value of $u(k) \in V_k$. Hence, by the union bound 
we have that for any fixed $u \in V_{k'}$
\begin{equation}
({\rm I}_k) \le	2^{k}\, \bP_{s_{n,z}} (\tau_u=k) . \label{eq: sum over levels2}
\end{equation}
With $\rho_n > 1 > \delta$, 
it is easy to verify 
that $j \mapsto \varphi_{n,k',h}(j)$ is strictly 
decreasing with $\varphi_{n,k',h}(k')=\rho_n (n-k') - h \ge 0$.
Further, 
for $b:=\varphi_{n,k',h}(k)$, $\wt{b}:=\varphi_{n,k',h}(k+1)$,  we get 
upon conditioning on $\eta_u(k)=y$, that   
	\begin{align}
&	\bP_{s_{n,z}} (\tau_u=k+1) \le \nn \\ 
& \qquad \sum_{i=0}^\infty 
	\bP_{s_{n,z}} ( \tau_u > k, 
	\eta_{u} (k) \in H_{b+i} )
	\sup_{y \in H_{b+i}} \bP_{\frac{y^2}{2}} (\eta_u (1) 
	\le \wt{b}) \,.
\label{eq: cond on height2}
	\end{align}
	Applying \cite[Lemma 4.6]{BK} at $p=q=1/2$ and 
	$\theta=\wt{b}^2/2 \le y^2/2$
		\begin{equation}
		\label{eq-I5.6}
		\sup_{y \in H_{b+i}} \bP_{\frac{y^2}{2}} (\eta_u (1)
		\le \wt{b} ) \leq 
		e^{-i^2/4}.
		\end{equation} 
Setting  $h'=\psi_{k',h}(k)$, we proceed to bound the first probability on the \abbr{rhs} 
of \eqref{eq: cond on height2}. To this end, recall  \eqref{eq:psi-pre-dom}, yielding that  
$\varphi_{n,k,h'}(j)  \le \varphi_{n,k',h}(j)$ for $j \in [0,k]$, with equality at $j=k$
(see \eqref{eq:barrier}--\eqref{psi-def}). Consequently, 
\begin{equation}
\bP_{s_{n,z}} (\tau_{u} > k, \eta_{u} (k) \in H_{b+i}) \leq  q_{n,k,z} (i;h'),
\end{equation}
for $q_{n,k,z}(\cdot;\cdot)$ of \Cref{lem:barrier-est}.  Since $n-k' \ge h$ and $\delta<1$, for any
$k < k'$,
\[
n-k - h' 
\ge k'-k - (k'-k)^\delta > 0 \,,
\] 
in which case, by \eqref{eq-oferlau1} we have that for any $i \in \bZ_+$
\begin{equation}\label{eq:oferlau2}
q_{n,k,z}(i;h') 
\le c_1 2^{-k} z_{h'} e^{- c_\ast z_{h'}} 
e^{ -z_{h'}^2/(8n)} \, g_{2n} (i+1) \,.
\end{equation}
Noting that  
$\sup_{n \ge 3} \{ g_{2n}(i+1) \} e^{-i^2/4}$ is summable (and $z_{h'} \ge z_h$),
we find upon combining \eqref{eq: sum over levels2}--\eqref{eq:oferlau2}, that for 
some $c_3$ finite and any $1 \le k < k'$,
\begin{align}\label{eq:final-ofla}
({\rm I}_{k+1}) &\le	2^{k+1} \sum_{i=0}^\infty q_{n,k,z}(i;h') e^{-i^2/4} \nn \\
&\le
c_3 (z_h+k_{k'}^\delta) e^{- c_\ast (z_h + k_{k'}^{\delta})} 
e^{-z_h^2/(8n)}  \,.
\end{align}
Further, with $\varphi_{n,k',h}(1) \le m_n-h$ we have similarly to 
\eqref{eq: sum over levels2}--\eqref{eq-I5.6} that  
\[
({\rm I}_{1}) \le	2 \bP_{s_{n,z}} (\tau_u=1) \le
	2 \bP_{\frac{x^2}{2}} (\eta_{u} (1) \le m_n-h) \le 2 e^{-z_h^2/4} \,,
\]
which is further bounded for $z_h \ge z \ge 1$ by the \abbr{rhs} of
\eqref{eq:final-ofla} at $k=0$ (possibly increasing the 
universal constant $c3$). Summing over $k \le k'$ it follows from 
\eqref{eq: upper bound split} and \eqref{eq:final-ofla} that for some 
universal $c_4<\infty$,
\begin{align*}
\bP_{s_{n,z}} (G_{n,k'}(h)) &\le c_3 e^{-c_\ast z_h} e^{-z_h^2/(8n)}
\sum_{k=0}^{k'} (z_h+k_{k'}^{\delta}) e^{-c_\ast k_{k'}^{\delta}} \\
&\le c_4 z_h  e^{ -c_{\ast} z_h} e^{-z_h^2/(8n)} \,,
\end{align*}
as claimed in \eqref{gee.2}.

\subsection{Comparing barriers: Proof of \Cref{lem-Gamma-Lambda}}
\label{zerosubsection}

Hereafter, let $\nu_{n,k,z}(\cdot)$ denote the 
finite measure on $[0,\infty)$ such that 
\begin{equation}\label{dfn:nu-z}
\nu_{n,k,z}(A) := 2^k \bP_{s_{n,z}}(\min_{j \le k} \{\hat \eta_v (j)\} > 0  \,, \; \hat \eta_{v} (k)  \in  A) \,,
\end{equation}
using the abbreviation 
$\nu_{n,z}=\nu_{n,n',z}$ (where $n'=n-\ell$). In view
of \eqref{eq-big-definition-sharp} 
and the Markov property of $\{\eta_v(j)\}$ at $j=n'$ we  
have that 
\begin{align}
\bE_{s_{n,z}}[\Lambda_{n,\ell}] = 
2^{n'} \bP_{s_{n,z}}(E_{n,\ell}(v)) 
 =  \int_{I_\ell} \ga_{n,\ell} ( y) \nu_{n,z}(dy) \,.
\label{equpperbd}
\end{align}
for $\ga_{n,\ell}(\cdot)$ of \eqref{35.1}. Similarly, setting the finite measure 
on $[-h_\ell,\infty)$
\begin{equation}\label{dfn:mu-z}
\mu_{n,z}(A) :=  2^{n'} \bP_{s_{n,z}}(\min_{j \le n'} 
\{\hat \eta_v (j) + \psi_{\ell}(j) \} > 0 \,, \; \hat \eta_{v} (n')  \in  A) \,,
\end{equation}
we have by the Markov property and \eqref{eq-big-definition} that 
\begin{align}
\bE_{s_{n,z}}[\Gamma_{n,\ell}] = 2^{n'} \bP_{s_{n,z}}(F_{n,\ell}(v)) 
 =  \int_{-h_\ell}^\infty \ga_{n,\ell} ( y) \mu_{n,z}(dy) \,.
\label{Markov-Gamma}
\end{align}
For $r \ge 0$,  recalling $H_r=[r,r+1]$, we decompose $\mu_{n,z}(H_r)-\nu_{n,z}(H_r)$ according to the possible 
values of 
$\tau := \max \{ j < n' :  \hat \eta_v(j) \le 0 \}$, 
to arrive at 
\begin{align}\label{eq:def-mu-nu}
\mu_{n,z}(H_r) - \nu_{n,z}(H_r) \le  2^{n'} \sum_{k=1}^{n'-1} {\bf p}_{n,k,z}(r) \,,  
\end{align}
for ${\bf p}_{n,k,z}(r)$ of  \Cref{lem-apriori-c}.
By \eqref{equpperbd}, \eqref{Markov-Gamma}, 
\eqref{eq:def-mu-nu} and the 
monotonicity of $y \mapsto \ga_{n,\ell}(y)$ we have that 
\begin{align}\label{amir-decmp}
\bE_{s_{n,z}}[\Gamma_{n,\ell} - \Lambda_{n,\ell}]  
\le& \sum_{r \notin I_\ell}  \ga_{n,\ell}(r) \mu_{n,z}(H_r) 
+ \sum_{r \in I_\ell} \ga_{n,\ell}(r)  \sum_{k=1}^{n'-1} 2^{n'} {\bf p}_{n,k,z} (r)  \nn \\
&  := \; {\sf I}_n (z,\ell) \; + \; {\sf II}_n (z,\ell) \,.
\end{align}
Dealing first with ${\sf I}_n(z,\ell)$ of \eqref{amir-decmp}, note that  
$\mu_{n,z}(H_r)=2^{n'} q_{n,n'}(r+h;h)$
for $q_{n,k}(i;h)$ of \Cref{lem:barrier-est} and $h=h_\ell$.  Combining \eqref{35.1} with
\eqref{eq-oferlau1} at $k=n'$ (where $k_n=\ell$), and having  $z+h_\ell \le 2z$ (as $z \to \infty$
before $\ell \to \infty$), yields  
for some $c_5$ finite, any $\ell \ge \hat \ell$, large $n$ 
and all $r \ge -h_\ell$
\begin{equation}\label{mun-bd}
\ga_{n,\ell}(r) \mu_{n,z} (H_r)  \leq  c_5 \,  z  e^{-c_\ast z}   \frac{(r+2h_\ell)^2}{\ell^{3/2}} e^{-r^2/(8 \ell)} 
e^{(r+h_\ell)^2/n} \,.  
\end{equation}
Substituting \eqref{mun-bd} in \eqref{amir-decmp} and taking $n \to \infty$ results with 
\begin{equation}\label{eq:neglig-In}
\varlimsup_{z \to \infty} \big\{ z^{-1} e^{c_\ast z} 
 \varlimsup_{n \to \infty} {\sf I}_n(z,\ell) \big\} \le  \vep_{\sf I}(\ell)  \,,
\end{equation}
where by our choice \eqref{dfn:I_ell} of $I_\ell$,
for any $\ell \to \infty$ 
\[
\vep_{\sf I}(\ell) := c_5
 \sum_{r \not\in I_\ell} \frac{(r+2 h_\ell)^2}{\ell^{3/2}} e^{-r^2/(8 \ell)}   \longrightarrow 0 \,. 
\]
In view of \eqref{35.1}, it suffices to show that for some $\vep_{\sf II}(\ell) \to 0$ 
and all $r \in I_\ell$,
\begin{equation}\label{eq:neglig-IIn}
\varlimsup_{z \to \infty} \big\{ z^{-1} e^{c_\ast z} 
 \varlimsup_{n \to \infty} \sum_{k=1}^{n'-1} 2^{n'}  {\bf p}_{n,k,z} (r) \big\}
\le \frac{\vep_{\sf II}(\ell)}{\sqrt{\ell}} \,  (r+1) e^{c_\ast r}  \,,
\end{equation}
in order to get the analog of \eqref{eq:neglig-In} for ${\sf II}_n(z,\ell)$ and
thereby complete the proof of the lemma.  
In view of \eqref{eq:bd-phi} 
we get 
\eqref{eq:neglig-IIn} upon showing that
\begin{equation}\label{eq:neglig-IIn-final}
\lim_{z \to \infty} \sup_{r \in I_\ell} \,
\varlimsup_{n \to \infty} 
\sum_{k=1}^{n'-1}  \psi_\ell(k)^3  
k_{n'}^{-3/2}  \exp\big(-\frac{z^2}{16k}-\frac{r^2}{4(n'-k)}\big) = 0 \,.
\end{equation}
Even without the exponential factor, since $\delta<\frac{1}{6}$
the sum in 
\eqref{eq:neglig-IIn-final}
over $\{ k : k_{n'}^\delta  \ge h_\ell \}$, where $\psi_{\ell}(k) \le 2 k_{n'}^\delta$
is bounded above by 
\begin{equation}\label{delta-under-16}
4 \sum_{k \ge h_\ell^{1/\delta}} k^{3\delta - 3/2} \le c\,  
{h_\ell^{(3-\frac{1}{2 \delta})}}_{\stackrel{\longrightarrow}{\ell \to \infty}} 0 \,.
\end{equation}
Further, the sum in \eqref{eq:neglig-IIn-final} over 
$\{ k : k_{n'}^\delta < h_\ell\}$ has 
$2 h_\ell^{1/\delta}$ terms,  which are 
uniformly bounded by $(2 h_\ell)^3\exp\big(-b_\ell/(4 h_\ell^{1/\delta})\big)$, where having 
\begin{equation}\label{dfn:b-ell}
b_\ell := \frac{z^2}{4} \wedge  \inf_{r \in I_\ell} \{ r^2 \} = \frac{\ell}{\log \ell} \,,
\end{equation}
makes that sum also negligible,  
as claimed in \eqref{eq:neglig-IIn-final}.
\qed

\subsection{Second moment: proof of \Cref{lem-second-moment}}
\label{firstsubsection}

In view of \eqref{def:Lambda-n} we have that 
\[
\bE_s [\Lambda_{n,\ell} (\Lambda_{n,\ell}-1)]
= \sum_{\stackrel{ u,v \in V_{n'}}{u \ne v}} \bP_s (E_{n,\ell}(u)  \cap E_{n,\ell}(v)  ) \,.
\]
We recall  the definition \eqref{eq-big-definition-sharp} of $E_{n,\ell}(\cdot)$ and 
split the preceding sum according to the values of 
$k'=|u \wedge v|  < n'
$ and $\hat \eta_v(k') > 0$. Specifically, 
having $2^{n'+k-1}$ such ordered pairs (for $k=n'-k'$), 
yields the bound
\begin{align}
\label{newequation73}
\bE_{s_{n,z}} [\Lambda_{n,\ell} (\Lambda_{n,\ell}-1)] &\le  \sum_{k=1}^{n'} 
2^{2k} \int_0^\infty \theta^2_{n,k,\ell} (y) \nu_{n,k',z}(dy)  =: \sum_{k=1}^{n'} {\sf J_{k}}
\end{align}
in terms of $\theta_{n,k,\ell}(\cdot)$ and 
$\nu_{n,k',z}(\cdot)$ of \eqref{eq:bd-theta} and \eqref{dfn:nu-z}, respectively. Further, the 
Markov property of $\eta_v(j)$ at $j=k'$ yields in terms of $q_{n,k',z}(\cdot;0)$ of
\eqref{eq-oferlau1-def}
\[
{\sf J_k} \le \sum_{i=0}^\infty 2^{k'} q_{n,k',z}(i;0) \big[ 2^k \sup_{y \in H_i} \theta_{n,k,\ell}(y) \big]^2 \,.
\]
Plugging into the preceding the bounds \eqref{eq:bd-theta} and \eqref{eq-oferlau1} (at $m=64 n
 \le n^2$),
we find that for some $c_6<\infty$, any $k \in (0,n')$, $z \ge 1$ and $n \ge n_0(\ell)$ as in 
\Cref{lem-apriori-c},
\begin{equation}\label{eq:bd-Jk}
{\sf J_k} \le 
c_6
e^{\beta_{n,k'}+2\beta_{n,k}} (1 \vee \ell/k)
 \frac{z}{\sqrt{k'_n}
 \wedge k'}  e^{- c_\ast z}  \, \sum_{i=0}^\infty (i+1)^3 e^{-c_\ast i}
\end{equation}
Using \eqref{eq:sum-beta}, $\beta_{n,k} \le 1$ and $k'_n \ge k_{n'}$ we get from \eqref{eq:bd-Jk} that 
for some $c_7<\infty$, 
\begin{equation}\label{eq:bd-Jk-main}
{\sf J_k} \le \left\{
\begin{array}{l}
 c_7 \, z \, e^{-c_\ast z} \,
k_{n'}^{-3/2},  \qquad  \quad  k_{n'} \ge \ell \,, \\
c_7 \, e^{-c_\ast z}  \,, \qquad \qquad \qquad k' < \ell \,, \\
c_7\,  z \, e^{-c_\ast z}\, 
\sqrt{\ell} k^{-3}, \quad  \quad \; \;  k < \ell \,,
\end{array} \right.
\end{equation}
where 
for $k < \ell$ we used the alternative bounds  $\beta_{n,k} \le 1 - \log k$ and 
$k'_n \ge \ell$.  Now, \eqref{eq:bd-Jk-main} implies that for all $n$,
\[
z^{-1} e^{c_\ast z} \sum_{k= \ell^{1/3}}^{n'} {\sf J}_k \le  
c_7 \, [ \,   \sqrt{\ell} \sum_{k = \ell^{1/3}}^{\ell-1} k^{-3} + \sum_{k=\ell}^{n'-\ell} k_{n'}^{-3/2} 
+ \sum_{k=n'-\ell}^{n'} z^{-1}   \, ] \le \delta_{\ell,z}
\]
for some $\delta_{\ell,z}
\to 0$, when $z \to \infty$ followed by $\ell \to \infty$.
Turning to control the remaining sum of ${\sf J}_k$ over $k < \ell^{1/3}$, 
note that by \eqref{dfn:b-ell}, upon comparing \eqref{35.1} and \eqref{eq:bd-theta}
we find that for some $\vep(\ell)
\to 0$ as $\ell \to \infty$ and any $n \ge n_0(\ell)$,
\[ 
\sum_{k=1}^{\ell^{1/3}} 2^k \sup_{y \ge 0} \{ \theta_{n,k,\ell}(y) \}  
\le 4^{\ell^{1/3}} \sup_{r \in I_\ell} \{  \ga_{n,\ell}(r) \}  \le \vep(\ell) \,.
\]
To complete the proof of \Cref{lem-second-moment}, recall 
\eqref{eq-big-definition-sharp}  and  \eqref{dfn:nu-z}, that for $k'=n'-k$ 
\begin{equation}\label{theta-first-mom}
2^k \int_0^\infty \theta_{n,k,\ell} (y) \nu_{n,k',z}(dy) = \bE_{s_{n,z}} [\Lambda_{n,\ell}] \,.
\end{equation}
Hence, we have on the \abbr{rhs} of \eqref{newequation73} that 
\[
\varlimsup_{z \to \infty} \big\{ z^{-1} e^{c_\ast z} 
\varlimsup_{n \to \infty} \sum_{k=1}^{\ell^{1/3}} {\sf J}_k  \big\} \le  \vep(\ell)  
\varlimsup_{z \to \infty} \big\{ z^{-1} e^{c_\ast z}
\varlimsup_{n \to \infty} \bE_{s_{n,z}} [\Lambda_{n,\ell}] \big\} \,,
\]
which together with \eqref{eq:2mom-lbd} and  \eqref{eq-right-tail-ex} imply that 
the \abbr{rhs} 
is for all $\ell$ large enough
at most $2 c \vep(\ell)$ (i.e. negligible, as claimed). \qed

\section{The Bessel process: proof of \Cref{prop-asymptotic-first-moment}}\label{sec-sharpbar}

Hereafter, set
$\la_\ell (y) := \frac{1}{2} (c_\ast \ell +y)^2$, $y\geq -c^\ast \ell$, with
\begin{equation}
\begin{aligned}
  \wh \ga_{\ell} (y) & :=   
  \bP_{[\la_\ell(y)]} (\eta^{\sharp}_{\ell}=0, \; \eta(0) \in c_\ast \ell + I_\ell)  \,, \\
  \wt \ga_{\ell}(y) & := \bE^\xi \big\{  \bP_{\xi(\la_\ell(y))} (\eta^{\sharp}_{\ell}=0,
  \; \eta(0) \in c_\ast \ell + I_\ell)   )\big\}  \,,
  \end{aligned}
 \label{3a.49}
\end{equation}
which are 
$\ga_{\infty,\ell}(\cdot)$ from \eqref{35.1}, restricted to $I_\ell$ of \eqref{dfn:I_ell},
and its regularization by an expectation, denoted $\bE^\xi$, over the independent Poisson($\la$) 
variable $\xi(\la)$ at $\la=\la_\ell(y)$.  We emphasize that the law of $\xi$ depends on $\ell$ but we suppress this from the notation. We follow this convention of suppressing
dependence in $\ell,n$ in many places throughout this section, e.g. in the definitions \eqref{dfn:B-kappa}, \eqref{eq-Ak}, \eqref{dfn-D-psi-T}, \eqref{dfn:Fpm_T}
and \eqref{eq-evening2-T} below.
Our goal here is to prove \Cref{prop-asymptotic-first-moment}, with
\begin{equation}
 \alpha_\ell :=\frac{1}{\sqrt{\pi \ell}}  \int_0^\infty y e^{c_{\ast} y}    \, \wt \ga_{\ell}(y) \,d y\,. \label{3a.54m}
 \end{equation}
In particular $\al_\ell <\infty$, since by standard Poisson tail estimates
for some $c<\infty$, 
\begin{equation}\label{bd-Pois-for-ofer}
\wt \ga_\ell(y) \le \bP(\sqrt{2\xi(\la_\ell(y))}  - c_\ast \ell  \in I_\ell) \le \exp\big\{-c \, {\rm dist}(y,I_\ell)^2
\big\}
\end{equation}
(see \cite[(3.8)]{BRZ}). 
Omitting hereafter from the notation  the (irrelevant) specific choice $v \in V_{n'}$, 
we recall from \eqref{def:eta-hat}, \eqref{dfn:nu-z} and \eqref{equpperbd} that 
\begin{align}\label{3a.48} 
2^{-n'} \bE_{s_{n,z}}[\Lambda_{n,\ell}] 
& = \bE_{s_{n,z}}
\Big( \wh \ga_{\ell}\big(\eta(n') -c_{\ast}\ell \big); \min_{j < n'} \{\hat \eta (j) \} > 0 
\Big). 
\end{align}
The first step towards \Cref{prop-asymptotic-first-moment} is our next lemma, 
utilizing the Markov structure from \cite[Lemma 3.1]{BRZ} to estimate the 
barrier probabilities on the \abbr{rhs} of \eqref{3a.48} via the law
$\bP_x^Y$ of a $0$-dimensional Bessel process $\{Y_t\}$, starting at $Y_1=x$.
To this end,  define for  $\kappa \in \bR$ the events
\begin{equation}\label{dfn:B-kappa}
B_\kappa :=\bigcap_{j=1}^{n'}  \left\{ Y_j > \bar\varphi_n (j) + \kappa \psi_\ell(j)  \right\},
\end{equation}
in terms of the barrier notations \eqref{eq:bar-st}, \eqref{psi-def}, \eqref{dfn:psi-ell}, and
associate to each $[0,1]$-valued $g(\cdot)$, the function 
\begin{equation}\label{dfn:wt-g}
\wt g(w) :=  \bE^\xi \big[ g\big(\sqrt{2\xi(w^2/2)}\big) \big] \,,
\end{equation}
so in particular $g(w)=\wh \ga_\ell(w-c_\ast \ell)$ yields $\wt g(w)=\wt \gamma_\ell(w - c_\ast \ell)$. 
 \begin{lemma}\label{theo-sbarrierd} 
There exist $U_s \overset{dist}{\Longrightarrow}  U_\infty$, 
\corOF{a centered Gaussian of variance $1/2$}, and $\vep_\ell \to 0$ as $\ell \to \infty$,
such that  for $s=s_{n,z}$, any
$[0,1]$-valued $g(\cdot)$ supported on $[\bar \varphi_n(n'),\infty)$  and $z 
\ge \ell$,
\begin{align}\label{18.01d}
 (1- \vep_\ell)^{-1}  \bE^{Y,s} 
 \big(
 \; \wt g(Y_{n'}) ; B_{-1} \big)  
&  \ge  \bE_{s}  \Big( g(\eta(n')) ; \min_{j < n'} \{\wh \eta(j)\}  > 0   \Big) 
\\
&  \ge (1-\vep_\ell)
\bE^{Y,s}
 \big(  \wt g (Y_{n'}); B_2  \big)  
 ,  \label{18.02d}
\end{align}
where $\bE^{Y,s}$ denotes expectation with respect to a $0$-dimensional Bessel process starting at
$Y_1= U_s + \sqrt{2s}$.
Further,  for some $\delta>0$,
\begin{equation}\label{eq-U-sub-gaussian}
\sup_s \, \{\, \bE [ e^{\delta U^2_s} ]\, \} \, < \infty \,.
\end{equation}
\end{lemma} 
\begin{proof} Recall from \cite[Lemma 3.1]{BRZ} the time in-homogeneous Markov chain 
\[
(\eta(0)=\corOF{\sqrt{2s}},Y_1,\eta(1),\cdots,Y_{n'},\eta(n'),\cdots) ,
\]
of law $\bQ_1^s$.
 From 
\cite[Lemma 3.1(a)]{BRZ} we have
that $\bQ[g(\eta(j))|Y_j]=\wt g(Y_j)$ of \eqref{dfn:wt-g},
and that $Y_1=\sqrt{2 \LL_1(s)}$ for a $\Gamma(s,1)$-random variable $\LL_1(s)$. 
Set $U_s := \sqrt{2 \LL_1(s)} - \sqrt{2s}$ and note that by \cite[Lemma 3.1(d,e)]{BRZ}, 
the random variables $\{\eta(j), j \ge 0\}$ and $\{Y_j, j \ge 1\}$ have respectively, 
the marginal laws $\bP_{s}$ and $\bP^{Y,s}$.

Standard 
large deviations for Gamma variables yield \eqref{eq-U-sub-gaussian} 
with $\delta<1/2$ (c.f. \cite[(3.13)]{BRZ}).
\corOF{Recall that $(\LL_1(s)-s)/\sqrt{2s} \overset{dist}{\Longrightarrow}  U_\infty$ when $s \to \infty$
(by the \abbr{clt}), hence the same convergence applies 
for $U_s = f_s((\LL_1(s)-s)/\sqrt{2s})$ and $f_s(\cdot)$ of \eqref{eq:f_s}.}
In addition,  setting for $k \in \bN$  the events 
\begin{equation}
\label{eq-Ak}
A_k := \bigcap_{j=0}^{k} 
\left\{ \eta(j) > \bar \varphi_n (j)  \right\}, 
\end{equation}
we have by the preceding and \eqref{dfn:B-kappa}
that the bound  \eqref{18.01d} follows from 
\begin{equation}\label{18.01-eff}
\bQ_{1}^{s} (\, g(\eta(n')); A_{n'} \cap B_{-1} ) \ge  (1-\vep_\ell) \bE_s  (\,  g(\eta(n')); A_{n'} )
\end{equation}
(taking $k=n'$ due to the assumed support of $g(\cdot)$ and including $j=0$ at no loss of 
generality since $z>0$).
Now, recall  from \cite[Lemma 3.1(b)]{BRZ} that 
\begin{align}\label{eq:prod}
\bQ_{1}^{s} (\, g(\eta(n'));\,A_{n'} \cap B_{-1})= \bE_s \big(\, g(\eta(n')) \prod_{j=1}^{n'} F^\eta_j ; \, A_{n'} \big) \,,
\end{align}
where if 
\begin{equation}\label{cond-A}
\sqrt{\eta^2(j-1)/2+\eta^2(j)/2} > \bar \varphi_n(j) \,,
\end{equation}
then by \cite[(3.14)]{BRZ},
\[
F^\eta_j := \bQ_1^s \big(Y_j > 
\bar \varphi_n(j) - \psi_\ell(j) \,\big|\,\eta(j-1),\eta(j)\, \big) \ge 1 -c  e^{-c \psi^2_\ell(j)} \,.
\]
Since  \eqref{cond-A} holds on the event $A_{n'}$  for $j=1,\ldots,n'$
\corON{recalling \eqref{dfn:psi-ell} that $\psi_\ell(j)=\psi_\ell(n'-j)$ and
splitting the product on the \abbr{lhs} of \eqref{eq:prod} to $j > n/2$ and $j \le n/2$,}
yields 
the inequality \eqref{18.01-eff}, and thereby \eqref{18.01d}, with
\begin{equation}\label{dfn:vep-ell}
\vep_\ell = 1 - \prod_{j=0}^\infty \big[1-c e^{- c (j^\delta+h_\ell)^2} \big]^{2}  \,,
\end{equation}
which converge to zero when $\ell \to \infty$. 

Recall from \cite[(3.4)]{BRZ} the notation $\bQ_2^{x^2/2}$ for the law of the Markov chain
$(Y_1,\eta(1),\cdots)$ started at $Y_1=x$.  To see
\eqref{18.02d}, we will show
\begin{equation}
\bQ_{2}^{x^2/2} (   \wt g(Y_{n'}) ;\,A_{n'-1} \cap B_2)
\geq (1-\vep_\ell) \bP^Y_x (\wt g (Y_{n'}) ;\,B_2 ) \,.
\label{120.5d}
\end{equation}
Taking the expectation over $x$ with respect to  the law of $Y_1$ under $Q_1^s$ and arguing as in   the proof of 
 \eqref{18.01d} will then give \eqref{18.02d}.
Turning to establishing \eqref{120.5d}, recall from \cite[Lemma 3.1(c)]{BRZ} that 
\[
\bQ_{2}^{x^2/2} (   \wt g(Y_{n'}) ;\,A_{n'-1} \cap B_2) = \bP_x^Y  (  \wt g(Y_{n'}) \prod_{j=1}^{n'-1} F^Y_j ; B_2) \,,
\]
where by \cite[(3.16)]{BRZ},
\[
F^Y_j := \bQ_2^{x^2/2}(\eta(j) > \bar \varphi_n(j) \,|\, Y_j, Y_{j+1}) \ge 1 - c e^{-c \psi^2_\ell(j)} \,,
\]
provided that
\begin{equation}\label{cond-B'}
\sqrt{Y_j Y_{j+1}} \ge \bar \varphi_n(j) + \psi_\ell(j)  \,.
\end{equation}
On $B_2$, we have  that for any $j < n'$, and all $\ell$ larger than some fixed universal constant,
\[
\sqrt{Y_j Y_{j+1}} > \bar \varphi_n(j+1) + 2 \psi_\ell(j+1) \ge \bar \varphi_n(j) + \psi_\ell(j) \,.
\]
In particular, with \eqref{cond-B'} holding on $B_2$ for any $j \in \{1,\ldots,n'-1\}$, by the same reasoning 
as before, this yields the inequality \eqref{120.5d}, and hence also \eqref{18.02d}. 
\end{proof}
We next estimate the barrier probabilities for 
$\{Y_j\}$
in terms of the law $\bP^W_x$ of a Brownian motion 
$\{W_t\}$, starting at  $W_1=Y_1=x$.  For \corOF{$0 \le T < T'  \le n'$, introduce the events
\begin{equation}\label{dfn-D-psi-T} 
D_{\pm 2 \psi,T,T'} := \{ W_t > \bar \varphi_n(t) \pm 2 \psi_\ell(t),  \;\forall t\in [T+1,T'] \},
\end{equation}
using hereafter $D_{ \kappa h_\ell,T,T'}$
if $\pm 2\psi_\ell(t)$ in \eqref{dfn-D-psi-T} is replaced by
the constant function $\kappa h_\ell$, with abbreviated notation
$D_{\pm 2 \psi,T}$ when $T'=n'-T$ and $D_{\pm 2 \psi}$ for $D_{\pm 2 \psi,0}$.} 
See Figure \ref{fig:DPsi} for a pictorial description. Recall the sets $B_\kappa$, see \eqref{dfn:B-kappa}.
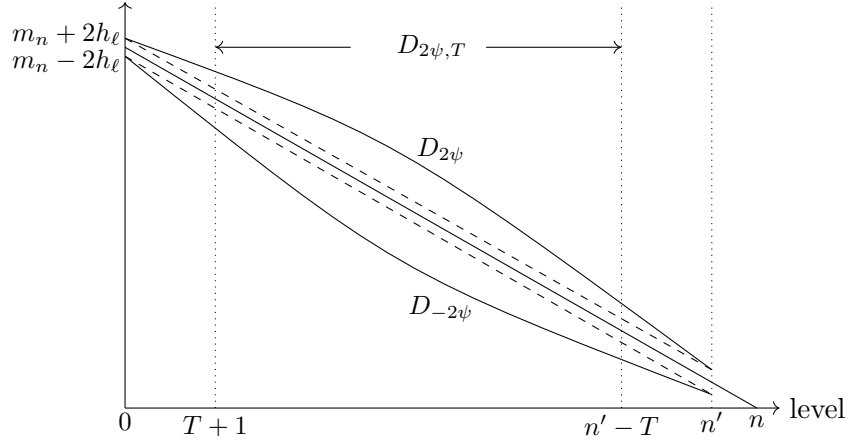
\begin{figure}
\begin{center}
\begin{tikzpicture}[scale=0.6]

\begin{scope}[shift={(-5,0)}]
    \draw [solid,->] (0.0,0) -- (14.5,0) node [ right]  {level};
    \draw [solid,->] (0,0) -- (0,9) node [below left] {};
    \draw[-] (0,7.8) .. controls (6,3) .. (13,0.3);
    \draw[-] (0,8.2) .. controls (6,6) .. (13,0.85);
    \draw[dashed] (0,7.8) -- (13,0.3);
    \draw[dashed] (0,8.2) -- (13,0.85);
  \end{scope}
   \begin{scope}[shift={(-5,0)}]
\draw[-] (0,8) -- (14,0);
 \node at (0,-0.3) {\small $0$};
 \node at (14,-0.3) {\small $n$};
 \node at (13,-0.3) {\small $n'$};
  \node at (-1.3,7.7) {\small $m_n-2h_\ell$};
  \node at (-1.3,8.3) {\small $m_n+2h_\ell$};
   \node at (2,-0.4) {\small $T+1$};
   \node at (11,-0.4) {\small $n'-T$};
     \draw [dotted] (2,0) -- (2,9);
    \draw [dotted] (11,0) -- (11,9);
     \draw [dotted] (13,0) -- (13,9);
     \node at (7,5.7) {\small $D_{2\psi}$};
      \node at (7,2.2) {\small $D_{-2\psi}$};
      \draw[{<[length=3,width=5]}-] (2,8) --(5,8);
      \draw[-{>[length=3,width=5]}] (8,8) -- (11,8) ;
      \node at (6.8,8) {\small $D_{2\psi,T}$};
  \end{scope}


\end{tikzpicture}
\end{center}
\caption{
 The curves in the events $D_{\pm 2h_\ell}$ (dashed lines) and $D_{\pm 2\psi}$ (curved, solid lines). The event $D_{2\psi,T}$ involves the curves between $T+1$ and $n'-T$.
}
\label{fig:DPsi}
\end{figure}
\begin{lemma}\label{theo-sbarrierbmd} 
For $\wh g(w):=\wt g(w)/\sqrt{w}$,
some $\vep_\ell \to 0$,
any
$\wt g(\cdot)$, $n,\ell$ as in \Cref{theo-sbarrierd} and all
$x>0$,
\begin{align}
\label{48.02d}
 \bE^Y_{x} \big(  \wt g (Y_{n'}); B_2   \big) &\ge 
  (1-\vep_{\ell} )  \sqrt{x} \, \bE^W_{x} \big( \wh g(W_{n'}) ; D_{2\psi}  \big),  \\
 \bE^Y_{x} \big( \wt g(Y_{n'}) ;  B_{-1} \big) 
\label{48.01d}
 & \leq (1-\vep_\ell)^{-1}
 \sqrt{x} \, \bE^W_{x} \big( \wh g(W_{n'}) ; D_{-2\psi}  \big). 
\end{align}
\end{lemma}
\begin{proof} Recall that up to the absorption time $\tau_\ast := \inf \{ t >1 : Y_t =0 \}$, the
$0$-dimensional Bessel process satisfies the \abbr{sde} 
\[
Y_{t}=W_{t}-\int_1^t \frac{1}{2Y_{s}} ds \,,  \qquad Y_1 = W_1 = x \,,
\]
with $\{W_{t}\}$ having the Brownian law $\bP^W_x$. Further, the event 
$\{Y_{n'}>0\}$ implies that $\{\tau_\ast > n'\}$, in which case by  Girsanov's theorem  
and monotone convergence, we have that for any bounded $\FF_{n'}$-measurable $Z$,
\be
\bE_{x}^{Y} ( Z;  Y_{n'}>0 )=\bE_{ x}^{W}
\Big(Z 
\sqrt{\frac{x}{W_{n'}}}
\,\,e^{-\frac{3}{8}\int_{1}^{n'} (W_{s})^{-2} ds}\,  ; \;
\inf_{t \in [1,n']}  \{W_{t} \} >0\Big).\label{Girs.1}
\ee
\corON{With $B_2$ containing the event corresponding to $D_{2 \psi}$  for the 
process $Y_t$, we} get \eqref{48.02d} by considering \eqref{Girs.1} for 
$Z=\wt g(W_{n'}) {\bf 1}_{D_{2\psi}}$. Indeed, the event $D_{2\psi}$  implies  
that $\inf_{t \le n'} \{ W_t - \bar \varphi_n(t) \} \ge 0$, hence
\begin{equation*}
e^{-\frac{3}{8} \int_{1}^{n'} (W_{s})^{-2} ds} \ge  
e^{-\frac{3}{8} \int_1^{n'} (n-s)^{-2} ds}  \ge 
e^{-\frac{3}{8 \ell}} := 1 -\vep_\ell 
\end{equation*}
with $\vep_\ell \to 0$ as $\ell \to \infty$.
Next, for all $\ell$ larger than some universal constant,
\[
\inf_{t \in [1,n']} \{ \bar \varphi_n(t) - 2 \psi_\ell(t)  \} = \bar \varphi_n(n') - 2 h_\ell > 0 \,,
\]
and with
$B'  :=  \{  W_j > \bar \varphi_n(j)  -  \psi_{\ell} (j), j = 1,2,\ldots,n'\}$,
it suffices for \eqref{48.01d} to show that  
\begin{align}
\bE^W_x (\wh g(W_{n'}); D_{-2\psi} \cap B' ) \ge (1 - \vep_{\ell} ) \bE^W_{x} ( \wh g(W_{n'});  B' ) \,.
\label{48.01n}
\end{align}
To this end,  since $\phi_t := \bar \varphi_n(t) - 2\psi_\ell(t)$ is a convex function, 
we get upon conditioning on $\{W_1,W_2,\ldots,W_{n'}\}$,  that 
\[
\bE^W_x (\wh g(W_{n'}) ; D_{-2\psi} \cap B') \ge \bE^W_x (\wh g(W_{n'}) \prod_{j=1}^{n'-1} F_j^W; B') \,,
\]
where by the reflection principle (see \cite[(2.1)]{BRZ} or \cite[Lemma 2.2]{Bramson1}), 
\begin{align}\label{eq:FW-bd}
F_j^W &:= \bP^W( \min_{u \in [0,1]} \{W_{j+u} - f_{\phi_j,\phi_{j+1}} (u;1)\} > 0  \,|\, W_j,W_{j+1}) \nn \\
& = 1 - \exp\big(-2 (W_j-\phi_j) (W_{j+1}-\phi_{j+1})\big) \,,
\end{align}
with $f_{a,b}(\cdot;1)$ denoting the line segment between $(0,a)$ and $(1,b)$. On the event
$B'$ we thus have that $F_j^W \ge 1-\exp(-2 \psi_\ell(j) \psi_\ell(j+1))$ for all 
$j \in \{1,2,\ldots,n'-1\}$,
thereby in analogy with \eqref{dfn:vep-ell}, establishing \eqref{48.01n} for 
\[
\vep_\ell = 1 - \prod_{j=0}^\infty  \big[1-e^{- 2 (j^\delta+h_\ell)((j+1)^\delta + h_\ell)} \big]^{2} \,,
\]
which converges to zero as $\ell \to \infty$.
 \end{proof}

 \subsection{ Proof of \Cref{prop-asymptotic-first-moment}}\label{sec-3.5}

Taking $s=s_{n,z}$ yields that 
$W_1=m_n + z + U_s$. For such $W_1$ let
\begin{equation}\label{dfn:alpha-n-pm}
\alpha^{(\pm)}_{n,\ell,z} := z^{-1} e^{c_\ast z} 2^{n'} 
\bE\big[ \sqrt{W_1/W_{n'}} \, \wt \ga_\ell(W_{n'} - c_\ast \ell) {\sf q}^{(\pm)}_{\wt n}
(W_1,W_{n'}) \, \big] \,,
\end{equation}
with $\wt n := n'-1$ denoting our barrier length and 
${\sf q}^{(\pm)}_{\wt n}(x,w) := {\sf q}^{(\pm)}_{\wt n,0}(x,w)$ for 
the corresponding non-crossing probabilities
\begin{equation}\label{dfn:Fpm_T}
{\sf q}^{(\pm)}_{\wt n,T}(x,w) := \bP^W_x(D_{\pm 2 \psi,T} \, | \, W_{n'}=w).
\end{equation}
Combining \eqref{3a.48} with Lemmas \ref{theo-sbarrierd} and \ref{theo-sbarrierbmd} for 
$g(\cdot)=\wh \ga_\ell(\cdot-c_\ast \ell)$ and $\wt g(\cdot) = \wt \ga_\ell (\cdot - c_\ast \ell)$, 
respectively, we have that 
\begin{align*}
 (1-\vep_\ell)^{-2}  \alpha^{(-)}_{n,\ell,z} \ge z^{-1} e^{c_\ast z} \bE_{s_{n,z}} [\Lambda_{n,\ell}] 
 \ge (1 \corOF{-} \vep_\ell)^{2} \alpha^{(+)}_{n,\ell,z} \,.
 \end{align*}
The proof of \Cref{prop-asymptotic-first-moment} thus amounts to showing that for
any $\epsilon > 0$ and all large enough $\ell$,
\begin{equation}\label{3a.52}
(1+\epsilon)^3 \alpha_\ell \ge 
\varlimsup_{z \to \infty}\varlimsup_{n \to \infty}
 \{ \alpha^{(-)}_{n,\ell,z} \}  
\ge 
\varliminf_{z \to \infty} 
\varliminf_{n \to \infty} \{ \alpha^{(+)}_{n,\ell,z} \} 
\ge (1- \epsilon)^3 \alpha_\ell \,.
\end{equation}
To this end, setting $z'=z+U_s$ and 
$W_{n'}=c_\ast \ell + y$, we write \eqref{dfn:alpha-n-pm} explicitly as
\[
\alpha^{(\pm)}_{n,\ell,z} =  \frac{e^{c_\ast z} 2^{n'}}{z \sqrt{2 \pi \wt n}} 
\bE \Big[ \int \wt \ga_\ell(y) dy
\frac{\sqrt{m_n+z'}}{\sqrt{c_\ast \ell+y}} {\sf q}_{\wt n}^{(\pm)} (m_n+z',c_\ast \ell+y) 
e^{-(m_n-c_\ast \ell+z'-y)^2/2 \wt n}  \Big] \,,
\]
with the expectation over $z'$.
Hereafter $\ell \le n/\log n$ so $|c_\ast - \rho_n| \ell \le 1$ and
$D_{\pm 2 \psi}$ imposes heights
$a^{(\pm)}=\rho_n \wt n + b^{(\pm)}$,
$b^{(\pm)} = c_\ast \ell \pm 2 h_\ell$ at barrier end points. Thus, in the 
preceding formula one needs only consider $z'$, $y \ge \pm 2 h_\ell$. 
Recall \eqref{eq-latenight}. With $m_n-c_\ast \ell = c_\ast n'- \vep_{n,n}$, upon setting 
$\Delta_n :=  \frac{1}{2 \wt n}(z'-y+c_\ast-\vep_{n,n})^2$,
we then get similarly to \eqref{eq:jay-ident1} that 
\begin{equation*}\label{eq:new-ident1}
\frac{1}{2\wt n} (c_\ast n' + z'-y - \vep_{n,n})^2 = (n'+1) \log 2 + c_\ast (z'-y - \vep_{n,n}) + \Delta_n \,.
\end{equation*}
Since $c_\ast \vep_{n,n} = \log n$, this
simplifies our formula for $\alpha^{(\pm)}_{n,\ell,z}$ to
\begin{align}\label{eq:alpha-pm-ident}
\alpha^{(\pm)}_{n,\ell,z} &=  \frac{1}{\sqrt{\pi \ell}} \int_{\pm 2 h_\ell}^\infty e^{c_\ast y} \wt \ga_\ell(y) 
  \frac{{\wh f}^{(\pm)}_{n,\ell,z}(y)}{\sqrt{1+y/(c_\ast \ell)}}  dy \,,\\
{\wh f}^{(\pm)}_{n,\ell,z}(y) & := \frac{n}{2 \sqrt{2} z} \bE \Big[ 
\frac{\sqrt{m_n+z'}}{\sqrt{c_\ast \wt n}} {\sf q}_{\wt n}^{(\pm)} (m_n+z',c_\ast \ell+y) e^{-c_\ast U_s} 
e^{-\Delta_n}   \Big] \,. \nn
\end{align}
By our uniform tail estimate \eqref{eq-U-sub-gaussian} for $U_s$ and the tail bound 
\eqref{bd-Pois-for-ofer} on $\wt \ga_\ell(y)$, up to an
error $\vep_n \to 0$ as $n \to \infty$, we 
can restrict the evaluation of ${\wh f}^{(\pm)}_{n,\ell,z}(y)$ to $|z'|+y \le C \sqrt{\log n}$. This forces
$m_n+z' = c_\ast \wt n (1 + \vep_{n})$ and eliminates $\Delta_n$, thereby allowing
us to replace  ${\wh f}^{(\pm)}_{n,\ell,z}(y)$ in \eqref{eq:alpha-pm-ident} by
\begin{equation}\label{eq:f-alt}
f^{(\pm)}_{n,\ell,z}(y) =  \bE \Big[ \frac{e^{-c_\ast U_s}}{\sqrt{2} z}  
\frac{\wt n}{2} {\sf q}_{\wt n}^{(\pm)} (m_n+z',c_\ast \ell+y)  \Big] \,.
\end{equation}
Recalling the events $D_{\kappa h_\ell,T}$, see below \eqref{dfn-D-psi-T}, we further consider the barrier probabilities 
\begin{equation}
  \label{eq-evening2-T}
 \wt {\sf q}_{\wt n,T}^{(\pm)} (x,w) := \bP^W_x (D_{\pm 2 h_\ell,T} \,| \, W_{n'} = w) \,,
\end{equation} 
using the 
abbreviated notation $\wt {\sf q}_{\wt n}^{(\pm)} (x,w) =\wt {\sf q}_{\wt n,0}^{(\pm)} (x,w)$.
\corON{Let  $ \bP_{x \to w}^{[t_1,t_2]}\(A_{m(t)}\)$ denote the probability that the Brownian bridge, taking 
the value $x$ at $t_1$ and  $w$ at $t_2$ remains above the barrier $m(t)$ on the interval $[t_1,t_2]$. 
Recall from \cite[Lemma 2.2]{Bramson1} that for a linear barrier $m(t)$,
\begin{align}\label{eq:lin-barrier-prob-idenBB}
\bP_{x\to w}^{[t_1,t_2]}\(A_{m(t)}\) =  1 - e^{-2 (x-m(t_1))_+(w-m(t_2))_+/(t_2-t_1)}  \,.
\end{align}
It follows that
\begin{align}\label{eq:lin-barrier-prob-iden}
\wt {\sf q}_{\wt n}^{(\pm)} (x,w) =\bP_{x\to w}^{[1,n']}\(A_{\bar \varphi_n(t)\pm 2 h_\ell}\) =  1 - e^{-2 (x-a^{(\pm)})_+(w-b^{(\pm)})_+/\wt n}  \,,
\end{align}
yielding for $x-a^{(\pm)}=z' \mp 2 h_\ell$ 
and $w-b^{(\pm)}= y \mp 2 h_\ell$ which are both $O(\sqrt{\log n})$,  
\begin{align}
\wt {\sf q}_{\wt n}^{(\pm)} (m_n+z',c_\ast \ell + y)  =
\frac{2 + \vep_n}{\wt n} (z' \mp 2 h_\ell) (y \mp 2 h_\ell)\,.
\label{eq:lin-barrier-prob}
\end{align}
Note further that
\begin{equation}\label{eq:Markov}
\wt{\sf q}_{\wt n,T}^{(\pm)} (x,w) = \bE_x^W \big[\bP_{W_{T+1}\to W_{n'-T}}^{[T+1,n'-T]}\(A_{\bar \varphi_n(t)\pm 2 h_\ell}\) | W_{n'} = w \big] \,.
\end{equation}}

The next lemma paraphrases \cite[Proposition 6.1]{Bramson1}
(with the proof given there also yielding the claimed uniformity).
 \begin{lemma} \label{lem-evening1}
For each $\epsilon>0$ there exist $T_\epsilon$, $n_\epsilon$ finite so that,
for any $\ell \ge 0$, $T \in [T_\epsilon, \frac{1}{2} \wt n]$, $x-a^{(\pm)},w-b^{(\pm)} \in [0,\log \wt n]$
and all $\wt n>n_\epsilon$ 
\begin{align}
(1-\epsilon) \,\wt {\sf q}^{(+)}_{\wt n, T}(x,w) &\le     {\sf q}^{(+)}_{\wt n,T}(x,w) 
  \le 
 {\sf q}^{(-)}_{\wt n,T}(x,w)  \leq (1+\epsilon) \,\wt {\sf q}^{(-)}_{\wt n,T}(x,w) \,.  
 \label{eq-evening3}
\end{align}  
\end{lemma}
Fixing $\epsilon>0$, we bound separately $f^{(\pm)}_{n,\ell,z}(y)$. Starting with $f^{(-)}_{n,\ell,z}(y)$,
we have from \eqref{eq:f-alt}, using the fact that $  
{\sf q}_{\wt n}^{(-)}\leq  
{\sf q}_{\wt n,T_\epsilon}^{(-)}$ and  the \abbr{rhs} of \eqref{eq-evening3}, that
\begin{eqnarray}\label{eq:f-alt-ub}
f^{(-)}_{n,\ell,z}(y) &\leq &  
 \bE \Big[ \frac{e^{-c_\ast U_s}}{\sqrt{2} z}  
\frac{\wt n}{2} {\sf q}_{\wt n,T_\epsilon}^{(-)} (m_n+z',c_\ast \ell+y)  \Big] 
\nonumber\\
&\leq &
(1+\epsilon) \bE \Big[ \frac{e^{-c_\ast U_s}}{\sqrt{2}z}  
\frac{\wt n}{2} \wt 
{\sf q}_{\wt n,T_\epsilon}^{(-)} (m_n+z',c_\ast \ell+y)  \Big]
\,.
\end{eqnarray}
Turning to evaluate $\wt{\sf q}_{\wt n,T}^{(\pm)} (m_n+z',c_\ast \ell+y)$, we get from 
\corON{ \eqref{eq:Markov}   and \eqref{eq:lin-barrier-prob-idenBB}
that 
\begin{equation}\label{eq:comp-core}
\wt{\sf q}_{\wt n,T}^{(\pm)} (m_n+z',c_\ast \ell+y)   \le
\frac{2} { \wt n-2T} \bE  \Big[ (\bar Z \mp 2 h_\ell)_+ (\bar Y \mp 2 h_\ell)_+ \Big] \,,
\end{equation}}
where $(\bar Z,\bar Y)$ follow the joint Gaussian distribution of 
\[
(W_{T+1} - \bar \varphi_n(T+1),W_{n'-T} - \bar \varphi_n(n'-T))
\]
given $W_1=m_n+z'$ and $W_{n'}=c_\ast \ell+y$.
It is further easy to verify that 
\begin{equation}
\Cov (\bar Z, \bar Y) = T \, \left[ \begin{matrix}  1  & 0 \\
 0 & 1 \end{matrix}\right]  
+ \frac{T^2}{\wt n}  
\left[ \begin{matrix}  -1  & 1 \\
 1 & -1 
\end{matrix}\right]  \,,
\end{equation}
and that 
\[
\bE [ (\bar Z, \bar Y) ] - [ z'+c_\ast + \frac{T}{\wt n} (y-z') , y + \frac{T}{\wt n} (z'-y) ] = o_{\wt n}(1) 
\]
\corOF{independently}
of $(z',y)$, decaying to zero when $\wt n \to \infty$ with $\ell,T$ kept fixed.
From this we get, in view of \eqref{eq:f-alt-ub} and \eqref{eq:comp-core}, that
\[
\varlimsup_{z \to \infty}\varlimsup_{n \to \infty}
\{ f^{(-)}_{n,\ell,z}(y) \}  \leq (1+\epsilon)  (y + 3 h_\ell)
\lim_{z,s \to \infty} \bE \big[ \frac{(z'+3 h_\ell)_+}{\sqrt{2}z}  e^{-c_\ast U_s} \big]
\]
provided $h_\ell \ge c_\ast + \sqrt{T_\epsilon/(2\pi)}$. 
In addition, with ${\mathcal B}=\{|U_s|<z/2\}$, or without such restriction, we get 
thanks to \eqref{eq-U-sub-gaussian}, via dominated convergence that  
\begin{equation}\label{eq:Us-limit}
\lim_{z,s \to \infty}  \bE \big[ {\bf 1}_{\mathcal B}
 \frac{(z' \pm 3 h_\ell)_+}{\sqrt{2} z} e^{-c_{\ast} U_s} \big]
=\frac{1}{\sqrt{2}} \bE \(e^{-c_{\ast}U_\infty}\) 
\corOF{= \frac{1}{\sqrt{2}} e^{c_\ast^2/4} } = 1 \,.
\end{equation}
\corOF{(Recall \Cref{theo-sbarrierd} that $U_\infty \sim N(0,1/2)$.)}
Combined with the previous display, we obtain
\begin{equation}\label{eq:bound-up}
\varlimsup_{z \to \infty}\varlimsup_{n \to \infty}
\{ f^{(-)}_{n,\ell,z}(y) \}  \leq (1+\epsilon)  (y + 3 h_\ell) \,.
\end{equation}
Note that for
$\ell \to \infty$
\begin{equation}
\label{eq-morning1}
\frac{1}{\sqrt{ \pi \ell}} \int_{- 2 h_\ell}^\infty e^{c_\ast y} \wt \ga_\ell(y) 
  \frac{y +3 h_\ell}{\sqrt{1+y/(c_\ast \ell)}}  dy \le \al_\ell (1+ \epsilon)\,.
\end{equation}
(Due to  \eqref{bd-Pois-for-ofer} the contribution to $\al_\ell$ outside 
$[\sqrt{\ell}/(2 r_\ell),2 r_\ell \sqrt{\ell}]$ is negligible, whereas within that interval
$y/\ell \to 0$ and $h_\ell/y \to 0$.) Combining  \eqref{eq:alpha-pm-ident},  \eqref{eq:bound-up} and  \eqref{eq-morning1} yields the \abbr{lhs} of 
\eqref{3a.52}, thereby completing the proof of the upper bound in
\Cref{prop-asymptotic-first-moment}.

Turning next to the lower bound on $f^{(+)}_{n,\ell,z}(y)$, 
we first truncate to $y \in [\sqrt{\ell}/(2r_\ell),\sqrt{\ell} \, 2 r_\ell]$ and restrict to
$z' \in [\frac{z}{2},\frac{3z}{2}]$ via the event ${\mathcal B}$.
Then, taking  $h_\ell \ge 2 (T+1)^\delta$ for 
$T=T_\epsilon$ of \Cref{lem-evening1}, guarantees that 
\[
3 h_\ell  \ge \sup_{t \in [1,T+1] \cap [n'-T,n'] } \{ 2 \psi_\ell(t) \},
\]
for $\psi_\ell(\cdot) \ge h_\ell$ of \eqref{dfn:psi-ell}. This in turn 
implies that \corOF{
\begin{align}\label{eq:event-bd}
D_{2\psi,T} &\subset  D_{2\psi}  
\bigcup (D_{3h_\ell,0,T+1}^c \cap D_{2 h_\ell,T})
\bigcup( D_{2 h_\ell,T} \cap D_{3h_\ell, \wt n - T,n'}^c ) \,,
\end{align}
where $D_{3h_\ell,T,T'}^c$ denotes the event of the Brownian motion 
crossing below the linear barrier in the definition of \corON{$D_{3h_\ell,T,T'}$,} see below \eqref{dfn-D-psi-T}.}
See figure \ref{fig:bessel} for an illustration of these events.
\begin{figure}
\begin{center}
\begin{tikzpicture}[scale=0.65]
\begin{scope}[shift={(-5,0)}]
    \draw [solid,->] (0.0,0) -- (14.0,0) node [ right]  {level};
    \draw [solid,->] (0,0) -- (0,9.8) node [below left] {};
  \draw [dotted] (13,0) -- (13,9);
   \draw [dotted] (2,0) -- (2,9);
    \draw [dotted] (11,0) -- (11,9);
    \draw[solid] (0,8.7)--(2,7.6);
    \draw[solid] (11,2.7)--(13,1.5);
    \draw[-] (0,8.2) .. controls (6,6) .. (13,0.85);
\draw[dashed] (0,8.9) .. controls (2,8.5) and  (4,7.5) .. (6,6.3);
\draw[dashed] (6,6.3) .. controls (8,4) and  (10,5) .. (13,2.6);
\draw[dotted,thick] (0,8.9) .. controls (0.5,8.4) and  (1,7) .. (2,8);
\draw[dotted,thick] (2,8) .. controls (4,9) and (5,4.5) .. (6,5.7);
\draw[dotted,thick] (6,5.7) .. controls (8,6) and  (10,4) .. (13,2.6);
  \end{scope}
   \begin{scope}[shift={(-5,0)}]
\draw[dotted] (0,8.2)--(13,0.85);
\draw[-] (2,7.06)--(11,1.97);
 \node at (13,-0.3) {\small $n'$};
 \node at (0,-0.3) {\small $1$};
  \node at (2,-0.3) {\small $T+1$};
   \node at (11,-0.3) {\small $n'-T$};
  \node at (-1.35,8.2) {\small $m_n+2h_\ell$};
   \node at (-1.45,9.2) {\small $m_n+z'$};
    \node at (-1.35,8.7) {\small $m_n+3h_\ell$};
   \node at (14.0,2.6) {\small $c^\ast \ell+y$};
    \node at (14.3,1.5) {\small $c^\ast \ell+3h_\ell$};
  \node at (14.3,0.8) {\small $c^\ast \ell+2h_\ell$};
  \end{scope}
\end{tikzpicture}
\end{center}
\caption{Depiction of the events $D_{2\psi}$ (dashed curve) and \corOF{$D_{3h_\ell,0,T+1}^c \cap D_{2 h_\ell,T}$} (dotted line).}
\label{fig:bessel}
\end{figure}
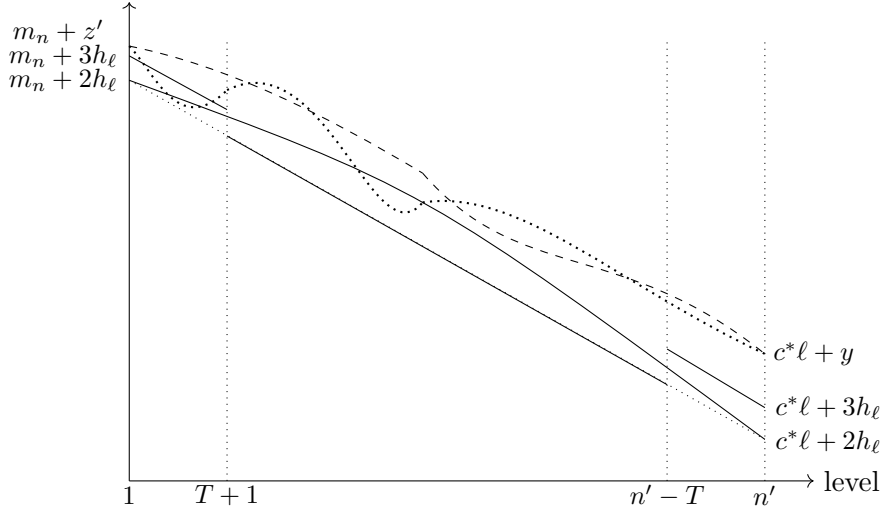
From \eqref{eq:event-bd} and the \abbr{lhs} of \eqref{eq-evening3}
we deduce by the union bound, 
that at $x=m_n+z'$, $w=c_\ast \ell + y$, for all $\wt n$ large enough
\begin{equation}\label{eq:basic-decomp}
{\sf q}^{(+)}_{\wt n} (x,w) \geq (1-\epsilon) \wt {\sf q}^{(+)}_{\wt n,T} (x,w) - 
\wt {\sf q}^{(\downarrow 3h_\ell,+)}_{\wt n,T} (x,w)
- \wt {\sf q}^{(+,\downarrow 3h_\ell)}_{\wt n,T} (x,w) \,,
\end{equation}
where $\wt {\sf q}^{(\downarrow 3h_\ell,+)}_{\wt n,T} (x,w)$ and $\wt {\sf q}^{(+,\downarrow 3h_\ell)}_{\wt n,T} (x,w)$ are the probabilities of the events 
$D_{3h_\ell,0,T+1}^c \cap D_{2 h_\ell,T}$ and $D_{2 h_\ell,T} \cap D_{3h_\ell, \wt n - T,n'}^c$
 under $\bP_x^W(\cdot|W_{n'}=w)$.

Proceeding to evaluating the latter terms,
note that conditional
on $(\bar Z,\bar Y)$ and the given values of $W_1=x$, $W_{n'}=w$,
the events $D_{3h_\ell,0,T+1}$, $D_{2 h_\ell,T}$ and 
$D_{3h_\ell,\wt n-T,n'}$ are mutually independent. Thus, setting
$y'=w - \rho_n \ell \in [y,y+1]$ and  assuming \abbr{wlog} that
$(\sqrt{\ell}/r_\ell) \wedge z \ge  8 h_\ell$, we have from 
\corON{\eqref{eq:lin-barrier-prob-idenBB}, that 
\begin{align*}
\bP(D_{3h_\ell,0, T+1}^c | \, \bar Z, \bar Y) &=1 - \bP_{x\to W_{ T+1}}^{[ 1, T+1]}\(A_{\bar \varphi_n(t)+3 h_\ell}\)= e^{-2(z'+\rho_n-3h_\ell)(\bar Z-3h_\ell)_+/T} \,, \\
\bP(D_{3h_\ell, \wt n - T,n'}^c | \, \bar Z, \bar Y) &=1 - \bP_{W_{n'-T }\to w}^{[n'-T ,n']}\(A_{\bar \varphi_n(t)+3 h_\ell}\)= e^{-2(\bar Y - 3h_\ell)_+(y'-3h_\ell)/T} \,.
\end{align*}}
Combining these identities with \eqref{eq:Markov} and the inequality \eqref{eq:basic-decomp}, we arrive at
\begin{align*}
{\sf q}^{(+)}_{\wt n} (x,w) \geq \bE \Big[  \big( 1-\epsilon -&  e^{-2(z'+\rho_n-3h_\ell)(\bar Z-3h_\ell)_+/T}
- e^{-2(y'-3h_\ell) (\bar Y - 3h_\ell)_+/T} \,  \big) \nn \\
 & \qquad  \qquad
 \corON{\bP_{\bar{Z}+\bar \varphi_n(T+1)\to \bar{Y}+\bar \varphi_n(n'-T) }^{[T+1,n'-T]}\(A_{\bar \varphi_n(t)+ 2 h_\ell}\)}  \Big] \,.
\end{align*}
The first \corOF{factor} on the \abbr{rhs} is at least $-2$ and for all $\ell$ larger than 
some universal $\ell_0(\epsilon)$ it exceeds 
$(1-\epsilon)^2$ on the event $A := \{ \bar Z \wedge \bar Y \ge 4 h_\ell \}$. 
Setting $V:= (\bar Z - 2 h_\ell)_+ (\bar Y - 2 h_\ell)_+$, we combine 
for the second term on the \abbr{rhs} the analog of 
identity \eqref{eq:lin-barrier-prob-iden} with the bound
$1-e^{-a} \in [a-a^2/2,a]$ on $\bR_+$ to arrive at
\[
(\wt n-2T) {\sf q}^{(+)}_{\wt n} (x,w) \ge 2
\bE \Big[ \big\{ (1-\epsilon)^2  - 2 {\bf 1}_{A^c}-\frac{2V}{\wt n - 2 T}\big\} V \Big] \,.
\]
Utilizing \eqref{eq:f-alt}, the uniform tail bounds one has on $(\bar Z-z',\bar Y-y)$ when $\wt n \to \infty$, 
for our truncated range of $z'$ and $y$, followed by \eqref{eq:Us-limit}, we conclude that
\begin{align*}
& \varliminf_{z \to \infty} \varliminf_{n \to \ff} 
\{ f^{(-)}_{n,\ell,z}(y) \}  \geq 
\varliminf_{z \to \infty} \varliminf_{n \to \ff}
 \bE \Big[ {\bf 1}_{\mathcal B} \frac{e^{-c_\ast U_s}}{\sqrt{2}z}  
\frac{\wt n}{2} {\sf q}_{\wt n}^{(+)} (m_n+z',c_\ast \ell+y)  \Big] \nn  \\
& 
\geq (1-\epsilon)^2 (y-3 h_\ell) \lim_{z,s \to \infty}
 \bE \Big[ {\bf 1}_{\mathcal B} \frac{(z'-3 h_\ell)_+}{\sqrt{2}z} e^{-c_\ast U_s} \Big]
 \ge (1-\epsilon)^2 (y - 3 h_\ell) \,.  
\end{align*}
Plugging this into \eqref{eq:alpha-pm-ident} and noting that for $\ell \to \infty$
\begin{equation*}
\frac{1}{\sqrt{\pi \ell}} \int_{\sqrt{\ell}/(2 r_\ell)}^{2 r_\ell \sqrt{\ell}}
 e^{c_\ast y} \wt \ga_\ell(y) 
  \frac{y - 3 h_\ell}{\sqrt{1+y/(c_\ast \ell)}} 
   dy \ge \al_\ell (1 - \epsilon)
\end{equation*}
we arrive at the \abbr{rhs} of \eqref{3a.52}, thereby completing the proof of
\Cref{prop-asymptotic-first-moment}.

\bibliographystyle{plain}

\bigskip
\noindent
\begin{tabular}{lll} 
& Amir Dembo\\
& Department of Mathematics \& Department of Statistics\\
& Stanford University, Stanford, CA 94305\\
& adembo@stanford.edu\\
& &\\
& &\\
& Jay Rosen\\
& Department of Mathematics\\
&  College of Staten Island, CUNY\\
& Staten Island, NY 10314 \\
& jrosen30@optimum.net\\
& &\\
& & \\
 & Ofer Zeitouni\\
& Faculty of Mathematics, 
 Weitzmann Institute and\\
 &Courant Institute, NYU\\
& Rehovot 76100, Israel and NYC, NY 10012 \\
& ofer.zeitouni@weizmann.ac.il
\end{tabular}

\end{document}